\documentclass[a4paper,11pt]{amsart}
\usepackage{amssymb}
\usepackage{mathrsfs,cleveref}
\usepackage{tikz}
 \usepackage[arrow,matrix]{xy}
\usepackage{amsmath,amssymb,amscd,bbm,amsthm,mathrsfs}
\usepackage{amsfonts,enumerate}
\usepackage{tcolorbox}
\usepackage{xcolor}
\newtheorem{thm}{Theorem}[section]
\newtheorem{lem}{Lemma}[section]
\newtheorem{cor}{Corollary}[section]
\newtheorem{prop}{Proposition}[section]
\newtheorem{rem}{Remark}[section]
\newtheorem{de}{Definition}[section]
\newtheorem{ex}{Example}[section]

\theoremstyle{definition}
\numberwithin{equation}{section}

\begin{document}

 \title[Rank-$3$ Generalized Clifford Manifold and Its Twistor Space]{
Rank-$3$ Generalized Clifford Manifold and Its Twistor Space}
\author[G.-Z. Ren]{Guangzhen Ren}
%
\address{%
Department of Mathematics\\ Zhejiang International Studies University\\ Hangzhou 310023, China
}
\email{gzren@zisu.edu.cn}
\thanks{The first author is supported by National Natural Science Foundation of China (No.~12101564) and Scientific Research Fund of Zhejiang Provincial Education Department (No.~Y202456224). The second author
is supported by Natural Science Foundation of Zhejiang Province (No.~LMS26A010006). The third author is supported by Taishan Scholars Program for Young Experts of Shandong Province (No.~tsqn202507265).}

\author[K. Tang]{Kai Tang}
%
\address{%
School of Mathematical Sciences, Zhejiang Normal University\\ Jinhua, Zhejiang 321004, China
}
\email{kaitang001@zjnu.edu.cn}

\author[Q. Wu]{Qingyan Wu}
\address{%
Department of Mathematics\\ Linyi University\\ Linyi, Shandong 276005, China
}
\email{qingyanwu@gmail.com}

\subjclass{Primary 53D18; Secondary 53C26, 53C28, 81T30}

\keywords{rank-3 generalized Clifford manifold, generalized complex structure, generalized hypercomplex structure, twistor space, generalized Nijenhuis tensor, Clifford rotation, T-duality}

\begin{abstract}
We introduce the notion of a rank-3 generalized Clifford manifold, defined by a triple of generalized complex structures satisfying Clifford-type relations. We show that every such structure canonically induces a generalized hypercomplex structure. We further describe a natural Spin(3)-action by Clifford rotations, which produces an $S^2 \times S^2$-family of generalized complex structures. The corresponding twistor space is then constructed, and we prove that the induced almost generalized complex structure is integrable. In contrast to the standard pure-spinor approach, the integrability of the twistor-space structure is established entirely in terms of the generalized Nijenhuis tensor. We further prove that the Clifford data underlying this twistor construction are compatible with T-duality, in the sense that T-duality preserves the rank-3 Clifford triple, the induced structures, and the associated Spin(3)-rotated family.
\end{abstract}

\thanks{The authors would like to express their sincere gratitude to Professor Hanfei for his detailed and enlightening introduction to generalized complex geometry and T-duality. They are also deeply grateful to Professor Wei Wang for his illuminating exposition of classical twistor theory and Clifford structures. Their generous guidance and valuable insights have played an important role in the development of this work.}
 \maketitle
\section{Introduction}
The development of generalized geometric structures has profoundly influenced modern differential geometry, particularly through its deep connections with string theory and supersymmetric sigma models. The foundational concept of generalized complex geometry was introduced by Nigel Hitchin \cite{MR2013140} and later systematically developed by Marco Gualtieri \cite{MR2811595}. By replacing the tangent bundle $TM$ with the generalized tangent bundle $TM \oplus T^*M$, equipped with the Courant bracket and a natural neutral pairing, this theory provides a unified framework encompassing both complex and symplectic geometry.
Within this framework, generalized K\"{a}hler geometry is defined by a pair of commuting generalized complex structures that are compatible with a common generalized metric (see \cite{MR2811595,MR201414}). This structure extends classical K\"{a}hler geometry and naturally appears in two-dimensional supersymmetric sigma models with torsion.

A further step in this direction is the extension toward hyperk\"{a}hler-type structures. In the classical setting, hyperk\"{a}hler geometry is governed by three complex structures satisfying quaternionic relations. Motivated by this, Andreas Bredthauer \cite{Bredthauer2007} introduced generalized hyperk\"{a}hler geometry by considering triples of generalized complex structures satisfying quaternionic-type algebraic relations. This framework allows one to describe supersymmetric backgrounds beyond the classical hyperk\"{a}hler setting \cite{EzhuthachanGhoshal2007}. Later, Kimura, Sasaki, and Shiozawa \cite{Kimura2022} proposed a broader formulation of generalized hyperk\"{a}hler structures based on the algebra of split-biquaternions. Their approach unifies aspects of hyperk\"{a}hler and bi-hypercomplex geometries and naturally appears in $\mathcal{N}=(4,4)$ supersymmetric nonlinear sigma models. More recently, Fino and Grantcharov \cite{FinoGrantcharov2025} studied generalized hypercomplex structures as $S^2$-families of anti-commuting generalized complex structures.
Moreover, further discussions on generalized twistor spaces of hypercomplex structures can be found in \cite{Deschamps} and \cite{GloverSawon2015}.

From a structural viewpoint, hypercomplex geometry on Courant algebroids was systematically formulated by Sti\'{e}non \cite{hyc}. In that work, hypercomplex structures are defined in terms of triples of orthogonal endomorphisms satisfying quaternionic relations, and it is proved that the vanishing of the associated Nijenhuis concomitants is equivalent to the existence of an (almost) torsion-free hypercomplex connection preserving the three complex structures. This result establishes a precise correspondence between integrability conditions and connection-theoretic characterizations. In \cite{HongStienon2015}, Wei Hong and Mathieu Sti\'{e}non proved that holomorphic symplectic structures and hypercomplex structures on Courant algebroids are equivalent notions.

More recently, hypercomplex geometry has been further broadened to include configurations in which the constituent complex structures satisfy Clifford-type rather than strictly quaternionic algebraic relations
\cite{MRclihma}. Such structures are closely related to bi-hypercomplex and bi-HKT (hyperk\"{a}hler with torsion) geometries and naturally arise in models involving both ordinary and mirror multiplets \cite{bihkt}. The associated Clifford algebras often decompose into direct sums of quaternionic components, reflecting a bisector-type geometric structure and connecting to the doubled target-space formalism in string theory.

In this paper, we introduce the notion of a {\it rank-$r$ almost generalized Clifford structure}, which consists of $r$ almost generalized complex structures
\[
\mathcal{I}_1,\mathcal{I}_2,\dots,\mathcal{I}_r
\]
satisfying the Clifford-type relations
\[
\mathcal{I}_i\mathcal{I}_j+\mathcal{I}_j\mathcal{I}_i=-2\delta_{ij}\mathrm{Id},
\qquad i,j=1,\dots,r.
\]
When each $\mathcal{I}_i$ is integrable, we call such a structure a {\it rank-$r$ generalized Clifford structure}.

In what follows, we focus exclusively on the rank-3 case. Unless explicitly stated otherwise, the term generalized Clifford manifold will always mean a rank-3 generalized Clifford manifold.
We will use the techniques in \cite{bihkt}, \cite{Kimura2022} and \cite{Bredthauer2007} to study the generalized Clifford structure. By the equivalence of the Clifford algebra $C l_{2,1}(\mathbb{R})$ and bi-quaternions, take
\begin{equation}\label{coincide00000}
\mathcal{J}_i:=\frac{1}{2} \epsilon_{ijk} \mathcal{I}_j \mathcal{I}_k, \quad \mathcal{G}:=-\mathcal{I}_1 \mathcal{J}_1=-\mathcal{I}_2 \mathcal{J}_2=-\mathcal{I}_3 \mathcal{J}_3=-\mathcal{I}_1 \mathcal{I}_2 \mathcal{I}_3 .
\end{equation}
We have
\begin{equation}\label{coincide}
\begin{array}{ll}
\mathcal{I}_i \mathcal{I}_j=-\delta_{i j}{\rm Id} +\epsilon_{i j k} \mathcal{J}_k, & \mathcal{J}_i \mathcal{J}_j=-\delta_{i j} {\rm Id}+\epsilon_{i j k} \mathcal{J}_k, \\
\mathcal{I}_i \mathcal{J}_j=-\delta_{i j} \mathcal{G}+\epsilon_{i j k} \mathcal{I}_k, & \mathcal{J}_i \mathcal{I}_j=-\delta_{i j} \mathcal{G}+\epsilon_{i j k} \mathcal{I}_k
\end{array}
\end{equation}
and
$$\mathcal{G}^{2}=1.$$

Our first main result shows that, in the rank-$3$ case, the individual integrability of the generalized Clifford generators implies the integrability of the induced generalized hypercomplex triple and controls the mixed generalized Nijenhuis concomitants: the off-diagonal mixed terms vanish, while the three diagonal terms coincide.
\begin{thm}\label{pullup001}
Let $\left(\mathcal{I}_1, \mathcal{I}_2, \mathcal{I}_3\right)$ be an almost generalized Clifford structure on $\mathbb{T} M:=T M \oplus T^* M$. Then
 \begin{equation}\label{ogcl1110106}
 \mathcal{N}\left(\mathcal{I}_i,\mathcal{I}_i\right)=0,\qquad i=1,2,3,
 \end{equation}
is equivalent to the following conditions:
\begin{equation}\label{ogcl11100}
\begin{aligned}
&\mathcal{N}\left(\mathcal{I}_i,\mathcal{I}_j\right)
=\mathcal{N}\left(\mathcal{J}_i,\mathcal{J}_j\right)=0,
\qquad i,j=1,2,3,\\
&\mathcal{N}\left(\mathcal{I}_i,\mathcal{J}_j\right)=0,
\qquad i\neq j,\\
&\mathcal{N}\left(\mathcal{I}_1,\mathcal{J}_1\right)
=\mathcal{N}\left(\mathcal{I}_2,\mathcal{J}_2\right)
=\mathcal{N}\left(\mathcal{I}_3,\mathcal{J}_3\right).
\end{aligned}
\end{equation}
Here $\mathcal{J}_i$ is the induced almost generalized complex structure defined by \eqref{coincide00000}.
\end{thm}
We further show that rank-$3$ generalized Clifford structures admit natural $Spin(3)$-transformations induced by orthogonal rotations. These transformations act as rank-$3$ Clifford rotations and generate an
 $S^2\times S^2$-family of generalized complex structures, extending the classical twistor sphere construction. Via stereographic projection, each Riemann sphere is identified with $S^2 \subset \mathbb{R}^3$. The associated rotation matrices $T\left(\zeta_1\right)$ and $S\left(\zeta_2\right) \in S O(3)$ encode the Clifford rotations determined by the parameters $\zeta_1$ and $\zeta_2$, which are the stereographic complex coordinates on the two factors of $S^2 \times S^2$. The explicit formulas for these matrices are provided in Section 4. So we establish a twistor transformation for generalized Clifford manifolds, showing that Clifford rotations generate a natural $S^2\times S^2$-family of generalized complex structures over $TM \oplus T^*M$.

\begin{thm}\label{6666661}
The transform
$$\left(\begin{array}{c}\mathcal{I}_1\\ \mathcal{I}_2\\ \mathcal{I}_3\end{array}\right)
\rightarrow\left(\begin{array}{c}\mathcal{K}_1(\zeta_1,\zeta_2)\\ \mathcal{K}_2(\zeta_1,\zeta_2)\\\mathcal{K}_3(\zeta_1,\zeta_2)\end{array}\right)
=\frac{1}{2}T(\zeta_1)\left(\begin{array}{c}\mathcal{I}_{1}+\mathcal{I}_{2}\mathcal{I}_{3} \\ \mathcal{I}_{2}+\mathcal{I}_{3}\mathcal{I}_{1}\\ \mathcal{I}_{3}+\mathcal{I}_{1}\mathcal{I}_{2}\end{array}\right)
+\frac{1}{2}S(\zeta_2)\left(\begin{array}{c}-\mathcal{I}_{1}+\mathcal{I}_{2}\mathcal{I}_{3} \\ -\mathcal{I}_{2}+\mathcal{I}_{3}\mathcal{I}_{1}\\- \mathcal{I}_{3}+\mathcal{I}_{1}\mathcal{I}_{2}\end{array}\right)$$
defines a rank-3 Clifford rotation. Moreover, the triple
$$\left(\mathcal{K}_1(\zeta_1,\zeta_2), \mathcal{K}_2(\zeta_1,\zeta_2), \mathcal{K}_3(\zeta_1,\zeta_2)\right)$$
 forms a new rank-$3$ generalized Clifford structure.
\end{thm}

Furthermore, by Theorem \ref{6666661}, and in close analogy with the generalized hyperk\"{a}hler case \cite{Bredthauer2007}, we construct the twistor space of a rank-$3$ generalized Clifford manifold. We then prove, in terms of the generalized Nijenhuis tensor, that the induced generalized complex structure on the twistor space is integrable.
\begin{thm}\label{lbnl1}
The twistor space of a rank-$3$ generalized Clifford manifold $M$ is the smooth manifold $Z=M \times S^2\times S^2$ endowed with the generalized complex structure $\mathbb{\hat{I}}$ defined at $T_{(p, \zeta_1,\zeta_2)} Z \oplus T_{(p, \zeta_1,\zeta_2)}^* Z$ as
$$
\left.\mathbb{\hat{I}}\right|_{(p, \zeta_1,\zeta_2)}=\left.\left.\hat{\mathcal{I}}(\zeta_1,\zeta_2)\right|_{T_p M \oplus T_p^* M} \oplus \mathcal{J}\right|_{T\left(S^2 \times S^2\right) \oplus T^*\left(S^2 \times S^2\right)},
$$
where $\hat{\mathcal{I}}(\zeta_1,\zeta_2)=\mathcal{K}_1(\zeta_1,\zeta_2)$ and
\[
\mathcal{J}:=
\left(\begin{array}{cccc}
J_{\zeta_1} & 0& 0& 0 \\
0&J_{\zeta_2} &0 & 0 \\
0& 0&-J^{T}_{\zeta_1} & 0 \\
0& 0&0 & -J^{T}_{\zeta_2}
\end{array}\right)
\]
is the generalized complex structure on $S^2\times S^2$ induced by the canonical complex structure via $S^2 \cong \mathbb{C} P^1$.
\end{thm}

The present construction is motivated by both differential geometry
and mathematical physics. The \(S^2\times S^2\)-family provides a
twistor-type parameter space for the two quaternionic sectors of a
rank-\(3\) generalized Clifford structure and offers a natural
framework for studying integrability and deformations of generalized
complex structures. Since the construction is formulated in terms of
the Dorfman bracket and the generalized Nijenhuis tensor, it naturally
accommodates \(B\)-field transformations and twisted brackets. Moreover,
the underlying Clifford data and the associated \(S^2\times S^2\)-family
behave naturally under \(T\)-duality. The positive and negative sectors are
also analogous to the two hypercomplex sectors appearing in generalized
hyperk\"ahler, bi-hypercomplex, and bi-HKT geometry, which are related
to supersymmetric sigma models with ordinary and mirror multiplets.

We also briefly comment on possible extensions to higher rank.
Although generalized Clifford structures are defined for arbitrary
rank, the present twistor construction relies on a feature specific
to rank \(3\): the involution
$\mathcal G=-\mathcal I_1\mathcal I_2\mathcal I_3$
produces two quaternionic sectors and hence the parameter space
\(S^2\times S^2\). In higher rank, the Clifford-module decomposition
depends essentially on the rank and does not in general admit the same
two-sector description. A higher-rank generalized twistor construction
would therefore require a separate analysis of the appropriate
parameter space and its integrability conditions. We leave this problem
for future work.

In Sect.~2, we investigate several structural properties of generalized complex structures in terms of the generalized Nijenhuis tensor. These results play a crucial role in proving the integrability of the generalized complex structure on the twistor space. In Sect.~3, we introduce generalized Clifford structures and present several basic examples. In Sect.~4, we construct the Spin(3)-transformations associated with rank-3 generalized Clifford structures and establish the generalized complex structure on the twistor space. In Sect.~5, we further study the behavior of rank-3 twisted generalized Clifford structures under T-duality, showing that T-duality preserves not only the original Clifford triple, but also the induced structures and the associated Spin(3)-rotated family.

\section{Preliminaries}

Let $M$ be a smooth manifold of even dimension $n$. Let $\mathbb{T} M=T M \oplus T^* M$, which is a direct sum of tangent bundle and cotangent bundle of $M$. The inner product $\langle\cdot,\cdot\rangle$ over $\mathbb{T} M$ is defined by
$$
\langle X+\xi, Y+\eta\rangle:=\frac{1}{2}(\xi(Y)+\eta(X)) ,
$$
where $X,Y\in TM$ and $\xi,\eta\in TM^* $. Obviously, the inner product $\langle\cdot,\cdot\rangle$ is  symmetric and has signature $(n, n)$.
The Dorfman bracket $[\cdot, \cdot]_D$ over  $\mathbb{T} M$ is defined by
\begin{equation}\label{dorfman1}
[X+\xi, Y+\eta]_D=[X, Y]+L_X \eta-\iota_Y {\rm d} \xi.
\end{equation}
The Dorfman bracket satisfies the Jacobi identity but is not skew-symmetric. The twisted version of a bracket $[X+\xi, Y+\eta]_D$ over $\mathbb{T} M$ with a $3$-form $H$ is defined by
\begin{equation}\label{dorfman2}
[X+\xi, Y+\eta]_{H}=[X, Y]+L_X \eta-\iota_Y {\rm d} \xi-\iota_Y \iota_X  H .
\end{equation}
The Jacobi identity for the bracket $[X+\xi, Y+\eta]_{H}$ requires that H must be closed.

\subsection{Generalized complex structure}
In this section, we formulate several properties of generalized complex structures purely in terms of the generalized Nijenhuis tensor, which will be used in the proofs of the subsequent theorems.

An almost generalized complex structure on $M$ is an orthogonal endomorphism $\mathcal{J}$ of $\mathbb{T} M$ satisfying $\mathcal{J}^2=-$ Id. If $\mathcal{J}$ is integrable with respect to the Dorfman bracket, then it is called a generalized complex structure.
For two bundle endomorphisms
$$
\mathcal{I}, \mathcal{J}: \mathbb{T} M \rightarrow\mathbb{T} M
$$
their mixed Nijenhuis tensor (also called the Nijenhuis concomitant) is the bilinear map
$$
\mathcal{N}(\mathcal{I}, \mathcal{J}): \Gamma(\mathbb{T} M) \times \Gamma(\mathbb{T} M) \rightarrow \Gamma(\mathbb{T} M)
$$
defined by
\begin{equation}\label{mnvf}
\begin{aligned}
\mathcal{N}(\mathcal{I},\mathcal{J})\left(A, B\right)= & \frac{1}{2}\left({\left[\mathcal{I} A,\mathcal{J} B\right]_D+\left[\mathcal{J} A, \mathcal{I} B\right]_D }
 -\mathcal{I}\left[A, \mathcal{J} B\right]_D-\mathcal{I}\left[\mathcal{J} A,B\right]_D\right. \\
& \left.-\mathcal{J}\left[A, \mathcal{I} B\right]_D-\mathcal{J}\left[\mathcal{I}A, B\right]_D+\mathcal{I}\mathcal{J}\left[A, B\right]_D+\mathcal{J}\mathcal{I}\left[A, B\right]_D\right) ,
\end{aligned}
\end{equation}
where $[\cdot, \cdot]_D$ is the Dorfman bracket and $A,B \in \Gamma(\mathbb{T} M)$. And the Nijenhuis tensor of $\mathcal{J}$ is
{\small \begin{equation}\label{mnvf111}
N_{\mathcal{J}}(A, B)=\mathcal{N}(\mathcal{J},\mathcal{J})\left(A, B\right)=[\mathcal{J} A, \mathcal{J} B]_D-\mathcal{J}[\mathcal{J}A, B]_D-\mathcal{J}[A, \mathcal{J} B]_D-[A, B]_D.
\end{equation}}

\begin{lem}\label{nea}
For generalized almost-complex structure $\mathcal{I}$ and $\mathcal{J}$, if $\mathcal{I}\mathcal{J}=-\mathcal{J}\mathcal{I}$, we have
\begin{itemize}
\item[(1)] The Nijenhuis expression for $\mathcal{I}\mathcal{J}$:
 \begin{equation*}\begin{split}
N_{\mathcal{I}\mathcal{J}}(A, B)= & \frac{1}{2}\left[N_{\mathcal{I}}(\mathcal{J}A, \mathcal{J} B)+N_{\mathcal{J}}(\mathcal{I} A, \mathcal{I} B)+N_{\mathcal{I}}(A, B)+N_{\mathcal{J}}(A, B)\right.\\
-&\left.\mathcal{I}  N_{\mathcal{J}}(\mathcal{I} A, B)-\mathcal{I} N_{\mathcal{J}}(A, \mathcal{I} B)
-\mathcal{J} N_{\mathcal{I}}(\mathcal{J} A, B)-\mathcal{J} N_{\mathcal{I}}(A, \mathcal{J} B)\right].
\end{split}\end{equation*}
\item[(2)] The mixed Nijenhuis expression for $\mathcal{I}$ and $\mathcal{J}$:
 \begin{equation*}\begin{split}
N\left(\mathcal{I},\mathcal{J}\right)(A, B)= & \frac{1}{2}\left[-N_{\mathcal{I}\mathcal{J}}(\mathcal{J}A, \mathcal{I} B)+\mathcal{I}\mathcal{J}\left(N_{\mathcal{I}}(A, B)-N_{\mathcal{J}}(A, B)\right)\right].
\end{split}\end{equation*}
\end{itemize}
\end{lem}
\begin{proof}
(1) Obviously, with respect to the Dorfman bracket \eqref{dorfman1}, we have
$$
[\mathcal{I} \mathcal{J} A, B]_D=-[\mathcal{J} \mathcal{I} A, B]_D
$$
since $\mathcal{I} \mathcal{J}=-\mathcal{J I}$. And the Dorfman bracket is linear in its first argument, i.e.,
$$
[A+B, C]_D=[A, C]_D+[B, C]_D.
$$
Then we have
\begin{equation}\label{o1}\begin{split}
&2N_{\mathcal{I}\mathcal{J}}(A, B)\\
= & [\mathcal{I}\mathcal{J} A, \mathcal{I}\mathcal{J} B]_D-\mathcal{I}\mathcal{J}[\mathcal{I}\mathcal{J} A, B]_D-\mathcal{I}\mathcal{J}[A, \mathcal{I}\mathcal{J} B]_D-[A, B]_D\\
+&[\mathcal{J}\mathcal{I} A, \mathcal{J}\mathcal{I} B]_D-\mathcal{J}\mathcal{I}[\mathcal{J}\mathcal{I} A, B]_D-\mathcal{J}\mathcal{I}[A, \mathcal{J}\mathcal{I} B]_D-[A, B]_D\\
=&\left(N_{\mathcal{I}}(\mathcal{J} A, \mathcal{J} B)+{\mathcal{I}[\mathcal{I}\mathcal{J}A,\mathcal{J}B]_D}+\mathcal{I}[\mathcal{J}A,\mathcal{I}\mathcal{J}B]_D+[\mathcal{J}A,\mathcal{J}B]_D\right)\\
+&\left(N_{\mathcal{J}}(\mathcal{I} A, \mathcal{I} B)+{\mathcal{J}[\mathcal{J}\mathcal{I}A,\mathcal{I}B]_D}+\mathcal{J}[\mathcal{I}A,\mathcal{J}\mathcal{I}B]_D+[\mathcal{I}A,\mathcal{I}B]_D\right)\\
-&{\mathcal{I}\mathcal{J}[\mathcal{I}\mathcal{J} A, B]_D}-\mathcal{I}\mathcal{J}[A, \mathcal{I}\mathcal{J} B]_D-[A, B]_D\\
-&{\mathcal{J}\mathcal{I}[\mathcal{J}\mathcal{I} A, B]_D}-\mathcal{J}\mathcal{I}[A, \mathcal{J}\mathcal{I} B]_D-[A, B]_D,
\end{split}\end{equation}
by the definition of $N_{\mathcal{I}}(\mathcal{J} A, \mathcal{J} B)$ and $N_{\mathcal{J}}(\mathcal{I} A, \mathcal{I} B)$.
And we have,
\begin{equation}\label{o200}\begin{split}
&\mathcal{I}[\mathcal{I}\mathcal{J}A,\mathcal{J}B]_D-\mathcal{J}\mathcal{I}[\mathcal{J}\mathcal{I} A, B]_D\\
=&\mathcal{I}\left(-[\mathcal{J}\mathcal{I}A,\mathcal{J}B]_D+\mathcal{J}[\mathcal{J}\mathcal{I} A, B]_D\right)\\
=&\mathcal{I}\left(-N_{\mathcal{J}}(\mathcal{I} A,B)-\mathcal{J}[\mathcal{I}A,\mathcal{J}B]_D-[\mathcal{I}A,B]_D\right),\\
\end{split}\end{equation}
by $\mathcal{I} \mathcal{J}=-\mathcal{J I}$ and the definition of $N_{\mathcal{J}}(\mathcal{I} A, B)$. Similarly,
\begin{equation}\label{o2}\begin{split}
&\mathcal{I}[\mathcal{J}A,\mathcal{I}\mathcal{J}B]_D-\mathcal{J}\mathcal{I}[ A, \mathcal{J}\mathcal{I}B]_D
=\mathcal{I}\left(-N_{\mathcal{J}}( A,\mathcal{I}B)-\mathcal{J}[\mathcal{J}A,\mathcal{I}B]_D-[A,\mathcal{I}B]_D\right),\\
&\mathcal{J}[\mathcal{J}\mathcal{I}A,\mathcal{I}B]_D-\mathcal{I}\mathcal{J}[\mathcal{I}\mathcal{J} A, B]_D=\mathcal{J}\left(-N_{\mathcal{I}}(\mathcal{J} A,B)-\mathcal{I}[\mathcal{J}A,\mathcal{I}B]_D-[\mathcal{J}A,B]_D\right),\\
&\mathcal{J}[\mathcal{I}A,\mathcal{J}\mathcal{I}B]_D-\mathcal{I}\mathcal{J}[A, \mathcal{I}\mathcal{J} B]_D=\mathcal{J}\left(-N_{\mathcal{I}}( A,\mathcal{J}B)-\mathcal{I}[\mathcal{I}A,\mathcal{J}B]_D-[A,\mathcal{J}B]_D\right).\\
\end{split}\end{equation}
Substituting \eqref{o200}-(\ref{o2}) into (\ref{o1}) yields
\begin{equation}\begin{split}
&2N_{\mathcal{I}\mathcal{J}}(A, B)\\
=&N_{\mathcal{I}}(\mathcal{J} A, \mathcal{J} B)+N_{\mathcal{J}}(\mathcal{I} A, \mathcal{I} B)-\mathcal{I} N_{\mathcal{J}}(\mathcal{I} A, B)-\mathcal{I} N_{\mathcal{J}}(A, \mathcal{I} B)
-\mathcal{J} N_{\mathcal{I}}(\mathcal{J} A, B)\\
-&\mathcal{J} N_{\mathcal{I}}(A, \mathcal{J} B)
+N_{\mathcal{I}}(A, B)+N_{\mathcal{J}}(A, B),
\end{split}\end{equation}
by $\mathcal{I} \mathcal{J}=-\mathcal{J I}$ and the definition of Nijenhuis tensor.\\
(2) By the definition \eqref{mnvf} and $\mathcal{I}\mathcal{J}=-\mathcal{J}\mathcal{I}$,
\begin{equation}\label{mnvf00}
\begin{aligned}
2\mathcal{N}(\mathcal{I},\mathcal{J})\left(A, B\right)= & {\left[\mathcal{I} A,\mathcal{J} B\right]_D+\left[\mathcal{J} A, \mathcal{I} B\right]_D }
 -\mathcal{I}\left[A, \mathcal{J} B\right]_D-\mathcal{I}\left[\mathcal{J} A,B\right]_D\\
-&\mathcal{J}\left[A, \mathcal{I} B\right]_D-\mathcal{J}\left[\mathcal{I}A, B\right]_D.
\end{aligned}
\end{equation}
By the definition of $\mathcal{N}_{\mathcal{I}\mathcal{J}}\left(\mathcal{J}A, \mathcal{I}B\right)$, $\mathcal{N}_{\mathcal{I}}\left(A, B\right)$ and $\mathcal{N}_{\mathcal{J}}\left(A, B\right)$, we have
\begin{equation}\label{dfcts1}\begin{split}
&{\left[\mathcal{I} A,\mathcal{J} B\right]_D+\left[\mathcal{J} A, \mathcal{I} B\right]_D }\\
=&-\mathcal{N}_{\mathcal{I}\mathcal{J}}\left(\mathcal{J}A, \mathcal{I}B\right)
+\mathcal{I}\mathcal{J}\left(\left[\mathcal{I} A, \mathcal{I} B\right]_D-\left[\mathcal{J} A, \mathcal{J} B\right]_D  \right)\\
=&-\mathcal{N}_{\mathcal{I}\mathcal{J}}\left(\mathcal{J}A, \mathcal{I}B\right)
+\mathcal{I}\mathcal{J}\left(\mathcal{N}_{\mathcal{I}}\left(A, B\right)+ \mathcal{I}\left[\mathcal{I} A, B\right]_D+ \mathcal{I}\left[ A, \mathcal{I}B\right]_D+\left[ A, B\right]_D\right)\\
&-\mathcal{I}\mathcal{J}\left(\mathcal{N}_{\mathcal{J}}\left(A, B\right)+ \mathcal{J}\left[\mathcal{J} A, B\right]_D+ \mathcal{J}\left[ A, \mathcal{J}B\right]_D+\left[ A, B\right]_D \right)\\
=&-\mathcal{N}_{\mathcal{I}\mathcal{J}}\left(\mathcal{J}A, \mathcal{I}B\right)
+\mathcal{I}\mathcal{J}\left(\mathcal{N}_{\mathcal{I}}\left(A, B\right)-\mathcal{N}_{\mathcal{J}}\left(A, B\right)\right)
+\mathcal{J}\left[\mathcal{I} A, B\right]_D+ \mathcal{J}\left[ A, \mathcal{I}B\right]_D\\
&+ \mathcal{I}\left[\mathcal{J} A, B\right]_D+ \mathcal{I}\left[ A, \mathcal{J}B\right]_D.
\end{split}\end{equation}
Substituting \eqref{dfcts1} into \eqref{mnvf00} yields
\begin{equation*}\label{mnvf0000}
\begin{aligned}
2\mathcal{N}(\mathcal{I},\mathcal{J})\left(A, B\right)= & -\mathcal{N}_{\mathcal{I}\mathcal{J}}\left(\mathcal{J}A, \mathcal{I}B\right)
+\mathcal{I}\mathcal{J}\left(\mathcal{N}_{\mathcal{I}}\left(A, B\right)-\mathcal{N}_{\mathcal{J}}\left(A, B\right)\right),
\end{aligned}
\end{equation*}
i.e., the second term of Lemma \ref{nea} holds.
\end{proof}
Then the following result in \cite{hyc} and \cite{HongStienon2015} is a direct corollary of the Lemma \ref{nea}.
\begin{lem}(\cite{hyc,HongStienon2015}). \label{nea000}For generalized complex structures $\mathcal{I}$ and $\mathcal{J}$, if $\mathcal{I}\mathcal{J}=-\mathcal{J}\mathcal{I}$, then
\begin{itemize}
\item[(1)]
$N_{\mathcal{I}\mathcal{J}}=0$;
\item[(2)]
$N\left(\mathcal{I},\mathcal{J}\right)=0$.
\end{itemize}
\end{lem}
\begin{proof}
(1) Because $N_{\mathcal{I}} = N_{\mathcal{J}} = 0$, the conclusion (1) follows immediately from the first term of Lemma~\ref{nea}.\\
(2)By $N_{\mathcal{I}} = N_{\mathcal{J}} = 0$, Lemma \ref{nea}--(2) and Lemma \ref{nea000}--(1), (2) holds clearly.

\end{proof}
\begin{lem}\label{pl11}
Let $\mathcal I,\mathcal J \in \Gamma(\mathrm{End}\,\mathbb{T} M)$ satisfy
\begin{equation}\label{I-J-cond}
\mathcal I^2=-\mathrm{Id}, \qquad
\mathcal I\mathcal J+\mathcal J\mathcal I=0.
\end{equation}
Then
\begin{equation}\label{beul11}
2\mathcal N(\mathcal I,\mathcal I\mathcal J)(A,B)=\mathcal N_{\mathcal I}(\mathcal J A,B)+\mathcal N_{\mathcal I}(A,\mathcal J B)+2\mathcal I\mathcal N(\mathcal I,\mathcal J)(A,B),
\end{equation}
where $A,B \in \Gamma(\mathbb{T} M)$.
\end{lem}

\begin{proof}
It follows from $\mathcal{I}^2 = -\mathrm{Id}$ and \eqref{mnvf111} that
\begin{equation}\label{transport}
[\mathcal I A,\mathcal I B]_D
=\mathcal N_{\mathcal I}(A,B)+
\mathcal I[\mathcal I A,B]_D
+\mathcal I[A,\mathcal I B]_D
+[A,B]_D,
\end{equation}
for all $A,B \in \Gamma(\mathbb{T} M)$.
Expanding $\mathcal{N}(\mathcal{I}, \mathcal{I}\mathcal{J})$ via the definition \eqref{mnvf}, we obtain
\begin{equation}\label{expand}
\begin{aligned}
&2\mathcal N(\mathcal I,\mathcal I\mathcal J)(A,B)\\
=&[\mathcal I A,\mathcal I\mathcal J B]_D+[\mathcal I\mathcal J A,\mathcal I B]_D
-\mathcal I[A,\mathcal I\mathcal J B]_D
-\mathcal I[\mathcal I\mathcal J A,B]_D
\\
-&(\mathcal I\mathcal J)[A,\mathcal I B]_D
-(\mathcal I\mathcal J)[\mathcal I A,B]_D
+\mathcal I(\mathcal I\mathcal J)[A,B]_D+(\mathcal I\mathcal J)\mathcal I[A,B]_D.
\end{aligned}
\end{equation}
Using $\mathcal I^2=-\mathrm{Id}$ and
$\mathcal I\mathcal J=-\mathcal J\mathcal I$, we obtain
\[
\mathcal I(\mathcal I\mathcal J)=-\mathcal J,
\qquad
(\mathcal I\mathcal J)\mathcal I=\mathcal J,
\]
so the last two terms in \eqref{expand} cancel.
Hence
\begin{equation}\label{reduced}
\begin{aligned}
2\mathcal N(\mathcal I,\mathcal I\mathcal J)(A,B)
&=[\mathcal I A,\mathcal I\mathcal J B]_D+[\mathcal I\mathcal J A,\mathcal I B]_D
-\mathcal I\left([A,\mathcal I\mathcal J B]_D+[\mathcal I\mathcal J A,B]_D\right)\\
&-\mathcal I\mathcal J\left([A,\mathcal I B]_D+[\mathcal I A,B]_D\right).
\end{aligned}
\end{equation}
Since $\mathcal I\mathcal J B=\mathcal I(\mathcal J B)$ and
$\mathcal I\mathcal J A=\mathcal I(\mathcal J A)$,
apply \eqref{transport} twice:
\begin{equation}\label{cor01}\begin{split}
[\mathcal I A,\mathcal I(\mathcal J B)]_D
&=\mathcal N_{\mathcal I}(A,\mathcal J B)+\mathcal I[\mathcal I A,\mathcal J B]_D
+\mathcal I[A,\mathcal I\mathcal J B]_D
+[A,\mathcal J B]_D, \\
[\mathcal I(\mathcal J A),\mathcal I B]_D
&=\mathcal N_{\mathcal I}(\mathcal J A,B)+\mathcal I[\mathcal I\mathcal J A,B]_D
+\mathcal I[\mathcal J A,\mathcal I B]_D
+[\mathcal J A,B]_D.
\end{split}\end{equation}
Substituting \eqref{cor01} into \eqref{reduced}
cancels the middle $\mathcal I$-terms and gives
\begin{equation}\label{almost}
\begin{aligned}
2\mathcal N(\mathcal I,\mathcal I\mathcal J)(A,B)
&=\mathcal N_{\mathcal I}(\mathcal J A,B)+\mathcal N_{\mathcal I}(A,\mathcal J B)+\mathcal I[\mathcal I A,\mathcal J B]_D
+\mathcal I[\mathcal J A,\mathcal I B]_D\\
&+[A,\mathcal J B]_D
+[\mathcal J A,B]_D-(\mathcal I\mathcal J)[A,\mathcal I B]_D
-(\mathcal I\mathcal J)[\mathcal I A,B]_D.
\end{aligned}
\end{equation}
By the definition of $\mathcal N(\mathcal I,\mathcal J)$, we have
\begin{equation}\label{mix}\begin{split}
[\mathcal I A,\mathcal J B]_D
+[\mathcal J A,\mathcal I B]_D
&=2\mathcal N(\mathcal I,\mathcal J)(A,B)+\mathcal I[A,\mathcal J B]_D
+\mathcal I[\mathcal J A,B]_D\\
&+\mathcal J[A,\mathcal I B]_D
+\mathcal J[\mathcal I A,B]_D.
\end{split}\end{equation}
Apply $\mathcal I$ to \eqref{mix} and use $\mathcal I^2=-\mathrm{Id}$:
\small{\begin{equation}\label{I-mix}
\begin{aligned}
\mathcal I[\mathcal I A,\mathcal J B]_D
+\mathcal I[\mathcal J A,\mathcal I B]_D
&=2\mathcal I\mathcal N(\mathcal I,\mathcal J)(A,B)-[A,\mathcal J B]_D
-[\mathcal J A,B]_D\\
&+(\mathcal I\mathcal J)[A,\mathcal I B]_D
+(\mathcal I\mathcal J)[\mathcal I A,B]_D.
\end{aligned}
\end{equation}}
Substituting \eqref{I-mix} into \eqref{almost}
shows
\[
2\mathcal N(\mathcal I,\mathcal I\mathcal J)(A,B)=\mathcal N_{\mathcal I}(\mathcal J A,B)+\mathcal N_{\mathcal I}(A,\mathcal J B)+2\mathcal I\mathcal N(\mathcal I,\mathcal J)(A,B),
\]
proving the claim.
\end{proof}

\subsection{Generalized real structure}

Motivated by the geometric meaning of integrability for generalized
complex structures, we introduce generalized real structures on a
Dorfman algebroid and investigate their integrability. The condition
$\mathcal G^{2}=\operatorname{Id}$ gives the algebraic decomposition
$\mathbb{T}M=\mathbb{T}M_{+}\oplus\mathbb{T}M_{-}$
into the $\pm1$ eigenbundles of $\mathcal G$. Integrability strengthens
this algebraic decomposition by requiring each eigenbundle to be
involutive under the Dorfman bracket. In this sense, an integrable
generalized real structure determines a decomposition of
$\mathbb{T}M$ into two complementary Dorfman-involutive subbundles.

\begin{de}
We call $\mathcal{G}$ a generalized almost-real structure if $\mathcal{G}^{2}=1$ and $\mathcal{G} \in O(n, n)$.
\end{de}
So the eigenspace of $\mathcal{G}$ w.r.t eigenvalue $1$ has the form
$$Ker(\mathcal{G}-{\rm Id})=\{A+\mathcal{G}A|\forall A\in\mathbb{T} M\}$$
and eigenspace of $\mathcal{G}$ w.r.t eigenvalue $-1$ has the form
$$Ker(\mathcal{G}+{\rm Id})=\{A-\mathcal{G}A|\forall A\in\mathbb{T} M\}.$$
We will show that the integrability is characterized by the vanishing of an associated Nijenhuis tensor. This provides a real analogue of generalized complex geometry and
gives a natural framework for decomposing the generalized tangent bundle into complementary Dorfman-involutive sectors.
\begin{de}
We call the generalized almost-real structure $\mathcal{G}$ a generalized real structure if $Ker(\mathcal{G}\pm{\rm Id})$ is closed with respect to the Dorfman bracket (\ref{dorfman1}), i.e.
\begin{equation}
[Ker(\mathcal{G}\pm{\rm Id}),Ker(\mathcal{G}\pm{\rm Id})]\subseteq Ker(\mathcal{G}\pm{\rm Id}).
\end{equation}
\end{de}
We define the Nijenhuis tensor of real structure $\mathcal{G}$ as
\begin{equation}\label{nittsrr}
N_{\mathcal{G}}(A, B)=[\mathcal{G}A, \mathcal{G} B]_D-\mathcal{G}[\mathcal{G}A, B]_D-\mathcal{G}[A, \mathcal{G}B]_D+[A, B]_D\quad (A,B\in\mathbb{T} M ),
\end{equation}
which differs slightly from the definition of the Nijenhuis tensor in the case of a generalized complex structure.
\begin{lem}
The vanishing of $N_{\mathcal{G}}$ is equivalent to the integrability of the real structure $\mathcal{G}$.
\end{lem}
\begin{proof}
For $A,B\in\mathbb{T} M$, we have
\begin{equation}
[A\pm\mathcal{G}A,B\pm\mathcal{G}B]_D=[A,B]_D\pm[A,\mathcal{G}B]_D\pm[\mathcal{G}A,B]_D+[\mathcal{G}A,\mathcal{G}B]_D.
\end{equation}
If $N_{\mathcal{G}}(A,B)=0$, then
\[
\mathcal{G}\left([A,\mathcal{G}B]_D+[\mathcal{G}A,B]_D\right)
=[A,B]_D+[\mathcal{G}A,\mathcal{G}B]_D,
\]
and hence
\[
\mathcal{G}[A\pm\mathcal{G}A,B\pm\mathcal{G}B]_D
=\pm[A\pm\mathcal{G}A,B\pm\mathcal{G}B]_D.
\]
Thus both eigenbundles are involutive.

Conversely, suppose that both eigenbundles are involutive. After complexification, write
$A=A_++A_-$ and $B=B_++B_-$ with
$\mathcal{G}A_\varepsilon=\varepsilon A_\varepsilon$ and
$\mathcal{G}B_\delta=\delta B_\delta$ for
$\varepsilon,\delta\in\{\pm1\}$. If $\varepsilon=\delta$, involutivity gives
$[A_\varepsilon,B_\varepsilon]_D$ in the same eigenspace, and a direct substitution into \eqref{nittsrr} gives
$N_{\mathcal{G}}(A_\varepsilon,B_\varepsilon)=0$. If
$\varepsilon=-\delta$, the same substitution gives
$N_{\mathcal{G}}(A_\varepsilon,B_\delta)=0$ identically. By bilinearity,
$N_{\mathcal{G}}=0$.
\end{proof}

The generalized metric, introduced by Hitchin and developed by Gualtieri,
is the analogue of a Riemannian metric in generalized geometry.
A generalized metric is an endomorphism
$\mathcal G \in \mathrm{End}(\mathbb{T} M)\cap O(n, n)$ such that
$\mathcal G^2 = \mathrm{Id}$,
and $\langle \mathcal{G} A, A\rangle > 0$ for all nonzero $A \in \mathbb{T} M$.
A generalized metric is
equivalent to a pair $(g,b)$ consisting of a Riemannian metric $g$ and a $2$-form $b$, and it has the specific expression
$$
\mathcal{G} =\left(\begin{array}{ll}
1 &0 \\
b & 1
\end{array}\right)\left(\begin{array}{cc}
0&g^{-1}  \\
g &0
\end{array}\right)\left(\begin{array}{cc}
1 &0 \\
-b & 1
\end{array}\right)=\left(\begin{array}{cc}
-g^{-1} b & g^{-1} \\
g-b g^{-1} b & b g^{-1}
\end{array}\right) .
$$
However, it is not integrable.
\begin{prop}The generalized metric $\mathcal{G}$, viewed as a generalized almost real structure, is in general not integrable.
\end{prop}
\begin{proof}
Let $\mathcal{G}$ be the generalized metric determined by a Riemannian metric $g$ and a 2 -form $b$. Its +1 eigenbundle is
$
C_{+}=\{X+(b+g)(X) \mid X \in T M\}.
$
The Dorfman bracket \eqref{dorfman1} satisfies the compatibility condition
$$
[A, A]_D=\mathcal{D}\langle A, A\rangle,
$$
where $\mathcal{D}$ is the Courant algebroid differential associated with the anchor map $\rho$, defined by $\langle\mathcal{D} f, A\rangle=\frac{1}{2} \rho(A) f$.
Taking $A=X+(b+g)(X) \in C_{+}$, we have
$$
[A, A]_D =\mathcal{D}\langle X+(b+g)(X), X+(b+g)(X)\rangle =\mathcal{D}(g(X, X)).
$$
If $\mathcal{G}$ were an integrable generalized real structure, then $C_{+}$ would be closed under the Dorfman bracket, hence
$$
[A, A]_D \in C_{+} \quad \text { for all } A \in C_{+} .
$$
But $\mathcal{D}(g(X, X))$ is a pure 1-form, while $C_{+} \cap T^* M=\{0\}$. Therefore one must have
$$
\mathcal{D}(g(X, X))=0,
$$
or equivalently,
$$
{\rm d}(g(X, X))=0 \quad \text { for all } X \in T M .
$$
This condition is incompatible with a nondegenerate metric $g$ in general. Hence the eigenbundles $C_{ \pm}$ of a generalized metric are not closed under the Dorfman bracket, and the associated almost real structure $\mathcal{G}$ is typically not integrable.
\end{proof}
\section{Generalized Clifford manifold }
A rank-3 Clifford-Hermitian manifold is a Riemannian manifold whose tangent bundle carries a pointwise representation of the rank-3 Clifford algebra \cite{MRclihma}. Concretely, it is equipped with three anticommuting orthogonal almost complex structures $I_1, I_2, I_3$, each compatible with the Riemannian metric. As a natural generalization, we introduce the generalized Clifford manifold by generalized complex structures on the Dorfman algebroid $T M \oplus T^* M$. This extension preserves many structural features of the classical theory while greatly enlarging the class of admissible geometries.
\begin{de}\label{gcmfd}
A (twisted) rank-$r$ generalized Clifford structure on $\mathbb{T}M$ consists of $r$ almost generalized complex structures
\[
I_1,\dots,I_r \in \mathrm{End}(\mathbb{T}M)
\]
satisfying the Clifford-type relations
 \begin{equation}\label{ogcl110}
\mathcal{I}_i \mathcal{I}_j+\mathcal{I}_j \mathcal{I}_i=-2 \delta_{i j} \mathrm{Id},
 \end{equation}
together with the integrability conditions
 \begin{equation}\label{ogcl111}
 \mathcal{N}(\mathcal{I}_{i},\mathcal{I}_{i})=0,\quad i=1,2,\cdots,r,
 \end{equation}
where $\mathcal N(\cdot,\cdot)$ denotes the generalized Nijenhuis concomitant associated with the (twisted) Dorfman bracket. A manifold endowed with such a structure is called a (twisted) rank-$r$ generalized Clifford manifold.
\end{de}
In the rank-$3$ case, Theorem~\ref{pullup001} shows that all off-diagonal mixed generalized Nijenhuis concomitants vanish, while the three diagonal concomitants $\mathcal N(\mathcal I_i,\mathcal J_i)$ coincide. In particular, the induced triple $(\mathcal J_1,\mathcal J_2,\mathcal J_3)$ is integrable and every rank-$3$ generalized Clifford structure canonically determines a generalized hypercomplex structure.\\
{\bf Proof of Theorem \ref{pullup001}}
($\Longrightarrow$) For $i\neq j$, the Clifford relations give
$\mathcal I_i\mathcal I_j=-\mathcal I_j\mathcal I_i$. Hence Lemma~\ref{nea000} yields
\begin{equation}\label{mixed001}
\mathcal{N}\left(\mathcal{I}_i,\mathcal{I}_j\right)=0,\qquad i\neq j.
\end{equation}
Together with \eqref{ogcl1110106}, this gives
\begin{equation}\label{ogcl111001}
\mathcal{N}\left(\mathcal{I}_i,\mathcal{I}_j\right)=0,\qquad i,j=1,2,3.
\end{equation}
For $i\neq j$, Lemma~\ref{nea000}--(1) also shows that
$\mathcal I_i\mathcal I_j$ is integrable. Since
$\mathcal J_k=\frac12\varepsilon_{kij}\mathcal I_i\mathcal I_j$, each
$\mathcal J_k$ is integrable. Moreover, the $\mathcal J_i$ satisfy the quaternionic relations, so Lemma~\ref{nea000} gives
\begin{equation}\label{ogcl111002}
\mathcal{N}\left(\mathcal{J}_i,\mathcal{J}_j\right)=0,\qquad i,j=1,2,3.
\end{equation}
By \eqref{coincide}, for $i\neq j$ one has
$\mathcal I_i\mathcal J_j=-\mathcal J_j\mathcal I_i$. Since both
$\mathcal I_i$ and $\mathcal J_j$ are integrable, Lemma~\ref{nea000} implies
\begin{equation}\label{t013}
\mathcal{N}\left(\mathcal{I}_i,\mathcal{J}_j\right)=0,\qquad i\neq j.
\end{equation}
It remains to compare the three diagonal concomitants. Set
\[
\Theta_i:=\mathcal N(\mathcal I_i,\mathcal J_i),\qquad i=1,2,3.
\]
For unit vectors $a=(a_1,a_2,a_3)$ and $b=(b_1,b_2,b_3)$ in
$\mathbb R^3$, define
\[
\mathcal I(a):=\sum_{i=1}^3a_i\mathcal I_i,\qquad
\mathcal J(b):=\sum_{i=1}^3b_i\mathcal J_i.
\]
By \eqref{ogcl111001} and \eqref{ogcl111002}, both
$\mathcal I(a)$ and $\mathcal J(b)$ are integrable generalized complex structures. Furthermore, \eqref{coincide} gives
\[
\mathcal I(a)\mathcal J(b)+\mathcal J(b)\mathcal I(a)
=-2(a\cdot b)\mathcal G.
\]
Thus, if $a\cdot b=0$, Lemma~\ref{nea000} yields
\[
0=\mathcal N(\mathcal I(a),\mathcal J(b))
=\sum_{i=1}^3a_i b_i\,\Theta_i,
\]
where \eqref{t013} has been used in the last equality. Taking
\[
a=\frac1{\sqrt2}(1,1,0),\qquad
b=\frac1{\sqrt2}(1,-1,0)
\]
gives $\Theta_1=\Theta_2$, and similarly $\Theta_1=\Theta_3$. Hence
\[
\mathcal N(\mathcal I_1,\mathcal J_1)
=\mathcal N(\mathcal I_2,\mathcal J_2)
=\mathcal N(\mathcal I_3,\mathcal J_3).
\]
Therefore \eqref{ogcl11100} holds.

($\Longleftarrow$) The converse follows immediately from the first line of
\eqref{ogcl11100}, by taking $i=j$.
\qed

{\it A generalized hypercomplex structure} \cite{FinoGrantcharov2025} on $M$ is a triple
$
(\mathcal{I}, \mathcal{J}, \mathcal{K})
$
of generalized complex structures on $\mathbb{T} M$ satisfying the quaternionic algebra:
$
\mathcal{I}^2=\mathcal{J}^2=\mathcal{K}^2=\mathcal{I} \mathcal{J} \mathcal{K}=-\mathrm{Id}.
$
 If each $\mathcal{I}, \mathcal{J}, \mathcal{K}$ is integrable in the generalized sense (i.e., its eigenbundle is Dorfman involutive), then $(M, \mathcal{I}, \mathcal{J}, \mathcal{K})$ is called a generalized hypercomplex manifold. Then by Lemma \ref{nea000}--(1), we have
\begin{prop}
The induced almost generalized complex structure $\mathcal{J}_{i}=\frac{1}{2} \epsilon_{ijk} \mathcal{I}_j \mathcal{I}_k$ in (\ref{coincide00000}) is integrable. Then the generalized Clifford manifold canonically determines a generalized hypercomplex structure $\left(\mathcal{J}_{1},\mathcal{J}_{2},\mathcal{J}_{3}\right)$.
\end{prop}

\begin{ex}{\it (Clifford--Hermitian manifolds).}
The rank-$3$ Clifford Hermitian manifold \cite{MRclihma} is a Hermitian manifold with three complex structure ${\rm I}_{a}$ $(a=1,2,3)$ such that
 \begin{itemize}
\item[(1)] $g({\rm I}_{a}X,{\rm I}_{a}Y)=g(X,Y)$, $X,Y\in TM$;
\item[(2)] $I_{a}I_{b}+I_{b}I_{a}=-2\delta_{ab}{\rm Id}.$
\end{itemize}
If the Clifford structures $I_i(i=1,2,3)$ are simultaneously integrable, that is,
$$
\mathcal{N}\left(I_i, I_j\right)=0, \quad i, j=1,2,3,
$$
then the associated generalized complex structures
$$
\mathcal{I}_i=\left(\begin{array}{cc}
I_i & 0 \\
0 & -I_i^*
\end{array}\right), \quad i=1,2,3,
$$
define a generalized Clifford structure on the generalized tangent bundle $\mathbb{T} M=T M \oplus T^* M$.
\end{ex}
\begin{ex}{\it (Generalized hyperk\"ahler manifolds).} Every generalized hyperk\"ahler manifold gives rise to a generalized Clifford structure.
This follows from the fact that the generalized hyperk\"{a}hler manifold admits three generalized K\"{a}hler structures $\left(\mathcal{I}_i,G\right)$ $(i=1,2,3)$ that satisfy the algebra of bi-quaternions $C l_{2,1}(\mathbb{R})$

\begin{equation}\label{rgk}
\begin{array}{ll}
\mathcal{I}_i \mathcal{I}_j=-\delta_{i j} {\rm Id}+\epsilon_{i j k} \mathcal{I}_k, & \tilde{\mathcal{I}}_i {\tilde{\mathcal{I}}_j}=-\delta_{i j}{\rm Id}+\epsilon_{i j k} \mathcal{I}_k, \\
\mathcal{I}_i \tilde{\mathcal{I}}_j=-\delta_{i j} {G}+\epsilon_{i j k} \tilde{\mathcal{I}}_k, & \tilde{\mathcal{I}}_i \mathcal{I}_j=-\delta_{i j} {G}+\epsilon_{i j k} \tilde{\mathcal{I}}_k,
\end{array}
\end{equation}
where  $\tilde{\mathcal{I}}_i=\mathcal{I}_i G$. It is direct to check that the three generalized complex structures $\tilde{\mathcal{I}}_i\ (i=1,2,3)$
form a generalized Clifford structure on $\mathbb{T} M$ .
\end{ex}

Obviously, by (\ref{coincide}) and (\ref{rgk}), we have
\begin{prop}
The generalized Clifford manifold $\left(M,\mathcal{I}_1,\mathcal{I}_2,\mathcal{I}_3\right)$ induces a generalized hyperk\"{a}hler structure when $\mathcal{G}=-\mathcal{I}_1 \mathcal{I}_2 \mathcal{I}_3 $ is a generalized metric.
\end{prop}

\begin{ex}{\it(Hyperk\"ahler manifolds).} Every hyperk\"ahler manifold yields a generalized Clifford structure in the generalized sense. In fact, take
$$\mathcal{I}_{i}:=\left(\begin{array}{cc}
0 & -\omega_{I_{i}}^{-1} \\
\omega_{I_{i}} & 0
\end{array}\right),\quad i=1,2,3.$$
It is direct to check that $\mathcal{I}_{i}\ (i=1,2,3)$ satisfies identity \eqref{ogcl110}-\eqref{ogcl111}, i.e., $\left(\mathcal{I}_{1},\mathcal{I}_{2},\mathcal{I}_{3}\right)$ is a generalized Clifford structure.
\end{ex}
\begin{rem} In the definition of a hyperk\"{a}hler manifold with hypercomplex structure  $I_1, I_2, I_3$, one does not need to impose the mixed Nijenhuis condition $N(I_i, I_j)=0\ (i\neq j)$ as an additional axiom. Indeed, for a hyperk\"{a}hler structure the complex structures $I_1, I_2, I_3$ are parallel with respect to the Levi-Civita connection, hence each of them is integrable. As a consequence, any mixed Nijenhuis concomitant vanishes automatically. Therefore the condition $N(I_i, I_j)=0\ (i\neq j)$ is redundant in the hyperk\"{a}hler setting.
\end{rem}

\section{The twistor space of a generalized Clifford manifold}

Take
\begin{equation}\label{56700}\mathcal{I}^{\pm}_i:=\frac{1}{2}\left(\mathcal{J}_i\pm\mathcal{I}_i\right)\ (i=1,2,3) \ \text{and}\ \mathcal{G}^{ \pm}:=\frac{1}{2}({\rm Id} \pm \mathcal{G}).\end{equation}
  It is direct to check that
\begin{lem}The endomorphisms $\mathcal G^{\pm}$ and $\mathcal I_i^{\pm}$ satisfy the following projection and quaternionic-type relations:
 \begin{equation}\label{567}\begin{split}
&(1)\ \left(\mathcal{G}^{ \pm}\right)^{2}=\mathcal{G}^{ \pm};\\
&(2)\ \mathcal{I}^{\pm}_i\mathcal{I}^{\pm}_j=-\delta_{ij}\mathcal{G}^{ \pm}+\varepsilon_{ijk}\mathcal{I}^{\pm}_k;\\
&(3)\  \mathcal{G}^{ \pm} \mathcal{G}^{ \mp}=\mathcal{G}^{ \pm} \mathcal{I}^{\mp}_i=\mathcal{I}^{ \pm}_i \mathcal{G}^{\mp}=\mathcal{I}^{\pm}_i\mathcal{I}^{\mp}_j=0.
\end{split}
 \end{equation}
\end{lem}

\begin{proof}
(1)
Using $\mathcal{G}^2={\rm Id}$ we compute
\[
(\mathcal{G}^\pm)^2=\frac14({\rm Id}\pm \mathcal{G})^2
=\frac14({\rm Id}\pm 2\mathcal{G}+\mathcal{G}^2)
=\frac14(2{\rm Id}\pm 2\mathcal{G})
=\frac12({\rm Id}\pm \mathcal{G})
=\mathcal{G}^\pm.
\]
(2) By \eqref{coincide} and \eqref{56700},
 \begin{equation}\begin{split}
 \mathcal{I}^{\pm}_i\mathcal{I}^{\pm}_j
=& \frac{1}{4}\left(\mathcal{J}_i \pm \mathcal{I}_i\right)\left(\mathcal{J}_j \pm\mathcal{I}_j\right)\\
 =& \frac{1}{4}\left(\mathcal{J}_i\mathcal{J}_j\pm \mathcal{J}_i\mathcal{I}_j \pm\mathcal{I}_i\mathcal{J}_j +\mathcal{I}_i\mathcal{I}_j\right)\\
  =&\frac{1}{4}\left[-\delta_{i j} {\rm Id} +\epsilon_{i j k} \mathcal{J}_k\pm (-\delta_{i j} \mathcal{G}+\epsilon_{i j k} \mathcal{I}_k) \right.\\
 &\left.\pm(-\delta_{i j} \mathcal{G}+\epsilon_{i j k} \mathcal{I}_k) +(-\delta_{i j} {\rm Id}+\epsilon_{i j k} \mathcal{J}_k)\right]\\
  =& \frac{1}{2}\left[-\delta_{i j} \left({\rm Id\pm \mathcal{G}}\right)+\epsilon_{i j k} \left(\mathcal{J}_k\pm\mathcal{I}_k\right)\right]\\
  =& -\delta_{ij}\mathcal{G}^{ \pm}+\varepsilon_{ijk}\mathcal{I}^{\pm}_k.
  \end{split}\end{equation}
  (3) Similarly, (3) holds by \eqref{coincide} and \eqref{56700}.
\end{proof}

\begin{cor}
Let $M$ be a generalized Clifford manifold. Since $\mathcal G^2=\mathrm{Id}$, one has the algebraic splitting
\begin{equation}
\mathbb{T} M=\mathbb{T}_{-1}\oplus\mathbb{T}_{1},
\end{equation}
where $\mathbb{T}_{-1}=\{A\in\mathbb{T}M\mid \mathcal GA=-A\}$ and
$\mathbb{T}_{1}=\{A\in\mathbb{T}M\mid \mathcal GA=A\}$.
Moreover, $\{\mathcal I^-_1,\mathcal I^-_2,\mathcal I^-_3\}$ and
$\{\mathcal I^+_1,\mathcal I^+_2,\mathcal I^+_3\}$ satisfy the quaternionic relations on
$\mathbb T_{-1}$ and $\mathbb T_1$, respectively. If $\mathcal G$ is integrable, then both eigenbundles are involutive under the Dorfman bracket.
\end{cor}
\begin{rem}
Recall that a generalized hypercomplex structure on a Courant algebroid consists
of three integrable generalized complex structures $(\mathcal J_1,\mathcal J_2,\mathcal J_3)$ satisfying the
quaternionic relations \cite{FinoGrantcharov2025}.
By contrast, the endomorphisms
\(\mathcal I_i^\pm\) are not generalized complex structures on all of
$\mathbb TM$, since $(\mathcal I_i^\pm)^2=-\mathcal G^\pm$
rather than \(-\operatorname{Id}\). However, on the eigenbundles
$\mathbb T_\pm=\ker(\mathcal G\mp\operatorname{Id}),$
the projections \(\mathcal G^\pm\) restrict to the identity, and hence
\[
\mathcal I_i^\pm\mathcal I_j^\pm
=
-\delta_{ij}\operatorname{Id}_{\mathbb T_\pm}
+\varepsilon_{ijk}\mathcal I_k^\pm.
\]
Thus the restrictions
$\left(
\mathcal I_1^\pm|_{\mathbb T_\pm},
\mathcal I_2^\pm|_{\mathbb T_\pm},
\mathcal I_3^\pm|_{\mathbb T_\pm}
\right)$
define quaternionic triples on the two eigenbundles.

If \(\mathcal G\) is integrable, then \(\mathbb T_\pm\) are involutive under the
Dorfman bracket, and the restricted triples remain quaternionic triples on the corresponding
Dorfman-involutive sectors. Without the integrability of \(\mathcal G\), they should be
regarded as algebraic generalized hypercomplex sectors.

This decomposition is analogous to the positive and negative
hypercomplex sectors appearing in generalized hyperk\"ahler and
bi-HKT geometry. The present construction is, however, more general,
since \(\mathcal G\) is not assumed to be a positive generalized metric.
\end{rem}
\begin{ex}
Let \(N\) be a hyperk\"ahler manifold endowed with hypercomplex structures
\(I,J,K\), and let
\begin{equation}\label{bcn001}
\mathcal J_I=
\begin{pmatrix}
I&0\\
0&-I^*
\end{pmatrix},
\qquad
\mathcal J_J=
\begin{pmatrix}
J&0\\
0&-J^*
\end{pmatrix},
\qquad
\mathcal J_K=
\begin{pmatrix}
K&0\\
0&-K^*
\end{pmatrix}
\end{equation}
be the associated generalized complex structures on $\mathbb TN=TN\oplus T^*N.$
On \(M=N\times N\), under the natural identification $\mathbb TM\cong\mathbb TN\oplus\mathbb TN,$
define
\begin{equation}\label{bcn002}
\mathcal I_1:=\mathcal J_I\oplus\mathcal J_I,
\qquad
\mathcal I_2:=\mathcal J_J\oplus\mathcal J_J,
\qquad
\mathcal I_3:=\mathcal J_K\oplus(-\mathcal J_K).
\end{equation}
Each \(\mathcal I_a\) is an integrable generalized complex structure,
and
\[
\mathcal I_a\mathcal I_b+\mathcal I_b\mathcal I_a
=-2\delta_{ab}\operatorname{Id}.
\]
Therefore, $(\mathcal I_1,\mathcal I_2,\mathcal I_3)$
defines a rank-\(3\) generalized Clifford structure. Moreover,
\[
\mathcal I_3\neq\mathcal I_1\mathcal I_2,
\]
so this example is not merely a generalized hypercomplex triple
written in Clifford notation.

We now describe explicitly the decomposition into the positive and
negative sectors. It follows from the definition of \(\mathcal J_a\)
in \eqref{coincide00000}, together with
\eqref{bcn001}--\eqref{bcn002}, that
\begin{equation}\label{coincide00000001}
\mathcal J_1
=\mathcal J_I\oplus(-\mathcal J_I),
\qquad
\mathcal J_2
=\mathcal J_J\oplus(-\mathcal J_J),
\qquad
\mathcal J_3
=\mathcal J_K\oplus\mathcal J_K.
\end{equation}
Consequently,
\[
\mathcal G=-\mathcal I_1\mathcal I_2\mathcal I_3
=\operatorname{Id}_{\mathbb TN}
 \oplus
 \bigl(-\operatorname{Id}_{\mathbb TN}\bigr).
\]
Hence the \(+1\) and \(-1\) eigenbundles of \(\mathcal G\) are
\[
\mathbb T_+
=\ker(\mathcal G-\operatorname{Id})
=\mathbb TN\oplus0,
\qquad
\mathbb T_-
=\ker(\mathcal G+\operatorname{Id})
=0\oplus\mathbb TN.
\]
In particular, both sectors are nonzero. The corresponding projections
are
\[
\mathcal G^+
=\frac12(\operatorname{Id}+\mathcal G)
=\operatorname{Id}_{\mathbb TN}\oplus0,
\qquad
\mathcal G^-
=\frac12(\operatorname{Id}-\mathcal G)
=0\oplus\operatorname{Id}_{\mathbb TN}.
\]
Using $\mathcal I_i^\pm=\frac12(\mathcal J_i\pm\mathcal I_i),$
we obtain the positive sector
\[
\mathcal I_1^+=\mathcal J_I\oplus0,
\qquad
\mathcal I_2^+=\mathcal J_J\oplus0,
\qquad
\mathcal I_3^+=\mathcal J_K\oplus0,
\]
and the negative sector
\[
\mathcal I_1^-=0\oplus(-\mathcal J_I),
\qquad
\mathcal I_2^-=0\oplus(-\mathcal J_J),
\qquad
\mathcal I_3^-=0\oplus\mathcal J_K.
\]
Thus $(\mathcal I_1^+,\mathcal I_2^+,\mathcal I_3^+)$ and $(\mathcal I_1^-,\mathcal I_2^-,\mathcal I_3^-)$
satisfy the quaternionic relations on
\(\mathbb T_+\) and \(\mathbb T_-\), respectively. This example
therefore exhibits two nontrivial quaternionic sectors, which may be
rotated independently in the subsequent \(S^2\times S^2\)-twistor
construction.
Moreover, this example admits a natural \(B\)-field deformation. For
any \(B\in\Omega^2(M)\), set \(H=dB\) and define
\[
\widetilde{\mathcal I}_a
=e^B\mathcal I_a e^{-B},
\qquad a=1,2,3.
\]
Then each \(\widetilde{\mathcal I}_a\) is integrable with respect to
the \(H\)-twisted Dorfman bracket, and
\[
\widetilde{\mathcal I}_a\widetilde{\mathcal I}_b
+
\widetilde{\mathcal I}_b\widetilde{\mathcal I}_a
=
-2\delta_{ab}\operatorname{Id}.
\]
Furthermore,
\[
\widetilde{\mathcal G}=e^B\mathcal G e^{-B},
\qquad
\widetilde{\mathcal G}^{\pm}=e^B\mathcal G^\pm e^{-B},
\qquad
\widetilde{\mathcal I}_a^\pm
=e^B\mathcal I_a^\pm e^{-B}.
\]
Hence the decomposition into the positive and negative sectors is
preserved under \(B\)-field transformations.
\end{ex}

We now apply the above decomposition to construct the twistor space of a generalized Clifford manifold.
Let
\[
\zeta_1 \longmapsto \frac{1}{1+|\zeta_1|^2}
\bigl(1-|\zeta_1|^2,\,-\mathbf{i}(\zeta_1-\bar{\zeta}_1),\,-(\zeta_1+\bar{\zeta}_1)\bigr)
\]
be the stereographic parametrization of the unit sphere \(S^2\), where \(\zeta_1\) denotes the standard complex coordinate on \(\mathbb{CP}^1\cong S^2\). Define
\begin{equation}\label{2201}
\begin{split}
T(\zeta_1)
:=
\frac{1}{1+|\zeta_1|^2}
\begin{pmatrix}
1-|\zeta_1|^2 & -\mathbf{i}(\zeta_1-\bar{\zeta}_1) & -(\zeta_1+\bar{\zeta}_1)\\
\mathbf{i}(\zeta_1-\bar{\zeta}_1) & 1+\frac{1}{2}\bigl(\zeta_1^2+\bar{\zeta}_1^2\bigr) & -\frac{\mathbf{i}}{2}\bigl(\zeta_1^2-\bar{\zeta}_1^2\bigr)\\
\zeta_1+\bar{\zeta}_1 & -\frac{\mathbf{i}}{2}\bigl(\zeta_1^2-\bar{\zeta}_1^2\bigr) & 1-\frac{1}{2}\bigl(\zeta_1^2+\bar{\zeta}_1^2\bigr)
\end{pmatrix}.
\end{split}
\end{equation}
Write \(t_{ij}\) for the entry in the \(i\)-th row and \(j\)-th column of \(T(\zeta_1)\). Then \(T(\zeta_1)\in \mathrm{SO}(3)\), and its first row is the unit vector of \(S^2\) determined by the stereographic coordinate \(\zeta_1\).

Similarly, let
\[
\zeta_2 \longmapsto \frac{1}{1+|\zeta_2|^2}
\bigl(1-|\zeta_2|^2,\,-\mathbf{i}(\zeta_2-\bar{\zeta}_2),\,-(\zeta_2+\bar{\zeta}_2)\bigr)
\]
be the stereographic parametrization of another copy of the unit sphere \(S^2\), where \(\zeta_2\) denotes the standard complex coordinate on \(\mathbb{CP}^1\cong S^2\). Define
\begin{equation}\label{2202}
\begin{split}
S(\zeta_2)
:=
\frac{1}{1+|\zeta_2|^2}
\begin{pmatrix}
1-|\zeta_2|^2 & -\mathbf{i}(\zeta_2-\bar{\zeta}_2) & -(\zeta_2+\bar{\zeta}_2)\\
\mathbf{i}(\zeta_2-\bar{\zeta}_2) & 1+\frac{1}{2}\bigl(\zeta_2^2+\bar{\zeta}_2^2\bigr) & -\frac{\mathbf{i}}{2}\bigl(\zeta_2^2-\bar{\zeta}_2^2\bigr)\\
\zeta_2+\bar{\zeta}_2 & -\frac{\mathbf{i}}{2}\bigl(\zeta_2^2-\bar{\zeta}_2^2\bigr) & 1-\frac{1}{2}\bigl(\zeta_2^2+\bar{\zeta}_2^2\bigr)
\end{pmatrix}.
\end{split}
\end{equation}
For convenience, we denote by \(s_{ij}\) the \((i,j)\)-entry of \(S(\zeta_2)\). Then \(S(\zeta_2)\in \mathrm{SO}(3)\), and its first row is the unit vector of \(S^2\) determined by the stereographic coordinate \(\zeta_2\).
Thus the pair \((\zeta_1,\zeta_2)\in \mathbb{CP}^1\times \mathbb{CP}^1\cong S^2\times S^2\) determines two unit vectors represented locally by the first rows of \(T(\zeta_1)\) and \(S(\zeta_2)\), respectively. The resulting first rotated generalized complex structure depends only on these two unit vectors and is therefore well defined on \(S^2\times S^2\).

Take
\begin{equation}\label{22601}\begin{split}&\hat{\mathcal{I}}^{+}_i=t_{i1}\mathcal{I}^{+}_1+t_{i2}\mathcal{I}^{+}_2+t_{i3}\mathcal{I}^{+}_3,\ (i=1,2,3),\\
&\hat{\mathcal{I}}^{-}_i=s_{i1}\mathcal{I}^{-}_1+s_{i2}\mathcal{I}^{-}_2+s_{i3}\mathcal{I}^{-}_3,\ (i=1,2,3).\end{split}\end{equation}
It is direct to check that
\begin{lem}\label{wke11}The quadruple $\left(\hat{\mathcal{I}}^{\pm}_1,\hat{\mathcal{I}}^{\pm}_2,\hat{\mathcal{I}}^{\pm}_3 ,\mathcal{G}^{ \pm}\right)$ satisfies the algebra of quaternion respectively, i.e.
 \begin{itemize}
\item[(1)] $\mathcal{\hat{I}}^{\pm}_i\mathcal{\hat{I}}^{\pm}_j=-\delta_{ij}\mathcal{G}^{ \pm}+\varepsilon_{ijk}\mathcal{\hat{I}}^{\pm}_k$;
\item[(2)] $ \mathcal{G}^{ \pm} \mathcal{G}^{ \mp}=\mathcal{G}^{ \pm} \mathcal{\hat{I}}^{\mp}_i=\mathcal{\hat{I}}^{ \pm}_i \mathcal{G}^{\mp}=\mathcal{\hat{I}}^{\pm}_i\mathcal{\hat{I}}^{\mp}_j=0$.
\end{itemize}
\end{lem}
\begin{proof}
(1) Denote the $i$-th row vector of matrix $T$ as
$$\tau_{i}=\left(t_{i1},t_{i2},t_{i3}\right),\quad i=1,2,3.$$
It is direct to check that their vector product satisfy
\begin{equation}\label{clock}
\tau_{1}=\tau_{2}\times\tau_{3},\ \tau_{2}=\tau_{3}\times\tau_{1},\ \tau_{3}=\tau_{1}\times\tau_{2},
\end{equation}
i.e.
\begin{equation}\label{clock2}
t_{1j}=t_{2r}t_{3s}\varepsilon_{rsj},\ t_{2j}=t_{3r}t_{1s}\varepsilon_{rsj},\ t_{3j}=t_{1r}t_{2s}\varepsilon_{rsj},\ j=1,2,3,
\end{equation}
where $\varepsilon_{i j k}$ is the Levi-Civita symbol
$$
\varepsilon_{i j k}=\left\{\begin{aligned}
+1, &\quad \text { if }(i, j, k) \text { is an even permutation of }(1,2,3), \\
-1, &\quad \text { if }(i, j, k) \text { is an odd permutation of }(1,2,3), \\
0, &\quad \text { if any two indices are equal.}
\end{aligned}\right.
$$
By the orthogonal property of $T$ and (\ref{clock2}), we have
\begin{equation}\begin{split}
\mathcal{\hat{I}}^{+}_i\mathcal{\hat{I}}^{+}_j
&=(\sum_{k=1}^{3}t_{ik}\mathcal{I}^{+}_k)(\sum_{l=1}^{3}t_{jl}\mathcal{I}^{+}_l)\\
&=\sum_{k,l=1}^{3}t_{ik}t_{jl}\mathcal{I}^{+}_k\mathcal{I}^{+}_l\\
&=\sum_{k,l=1}^{3}t_{ik}t_{jl}\left(-\delta_{kl}\mathcal{G}^{+}+\varepsilon_{kls}\mathcal{I}^{+}_s\right)\\
&=-\delta_{ij}\mathcal{G}^{ +}+\varepsilon_{ijs}\mathcal{\hat{I}}^{+}_s.
\end{split}\end{equation}
Similarly, we have $\mathcal{\hat{I}}^{-}_i\mathcal{\hat{I}}^{-}_j=-\delta_{ij}\mathcal{G}^{ -}+\varepsilon_{ijk}\mathcal{\hat{I}}^{-}_k$. \\
(2) Obviously, conclusion (2) holds by the third equation (\ref{567}).
\end{proof}
Take
\begin{equation}\label{ttnnb}
\mathcal{K}_i(\zeta_1,\zeta_2):=\hat{\mathcal{I}}^{+}_i+\hat{\mathcal{I}}^{-}_i=\sum_{l=1}^{3}\left( t_{il}\mathcal{I}_l^{+}+s_{il}\mathcal{I}_l^{-}\right),\ i=1,2,3.\end{equation}
For simplicity, we denote $\mathcal{K}_i(\zeta_1,\zeta_2)$ by $\mathcal{K}_i$.
Then by \eqref{56700},\eqref{2201}, \eqref{2202} and \eqref{22601}, \eqref{ttnnb} can be rewritten as
$$\left(\begin{array}{c}\mathcal{K}_1\\ \mathcal{K}_2\\\mathcal{K}_3\end{array}\right)
=\frac{1}{2}T\left(\begin{array}{c}\mathcal{I}_{1}+\mathcal{I}_{2}\mathcal{I}_{3} \\ \mathcal{I}_{2}+\mathcal{I}_{3}\mathcal{I}_{1}\\ \mathcal{I}_{3}+\mathcal{I}_{1}\mathcal{I}_{2}\end{array}\right)
+\frac{1}{2}S\left(\begin{array}{c}-\mathcal{I}_{1}+\mathcal{I}_{2}\mathcal{I}_{3} \\ -\mathcal{I}_{2}+\mathcal{I}_{3}\mathcal{I}_{1}\\- \mathcal{I}_{3}+\mathcal{I}_{1}\mathcal{I}_{2}\end{array}\right).$$
Now we can prove Theorem \ref{6666661}.\\
{\it Proof of Theorem \ref{6666661} }
By the definition of $\mathcal{K}_i$ in (\ref{ttnnb}) and Lemma \ref{wke11}, one checks directly that

\begin{equation}\label{ttnnb00}
\mathcal{K}_i^2=-\operatorname{Id}(i=1,2,3), \quad \mathcal{K}_i \mathcal{K}_j=-\mathcal{K}_j \mathcal{K}_i(i \neq j),
\end{equation}
and
\begin{equation}\label{ttnnb01}
\left\langle\mathcal{K}_i \cdot, \mathcal{K}_i \cdot\right\rangle=\langle\cdot, \cdot\rangle .
\end{equation}
Hence $\left(\mathcal{K}_1, \mathcal{K}_2, \mathcal{K}_3\right)$ forms an almost generalized Clifford structure.

We now prove that these three structures $\mathcal{K}_1, \mathcal{K}_2, \mathcal{K}_3$ are simultaneously integrable.
Define
$$
\alpha_i^{l}:=\frac{1}{2}\left(t_{il}-s_{il}\right), \quad \beta_i^l:=\frac{1}{2}\left(t_{il}+s_{il}\right).
$$
With this notation, $\mathcal{K}_i$ can be written as the linear combination
$$
\mathcal{K}_i=\sum_{l=1}^{3} \left(\alpha_i^l\mathcal{I}_l+\beta_i^l \mathcal{J}_l \right).
$$
By the bilinearity of Nijenhuis tensor, we have
\begin{equation}\label{t01}\begin{split}
\mathcal{N}\left(\mathcal{K}_i,\mathcal{K}_j\right)=\sum_{l,m=1}^{3}& \left(\alpha_i^l \alpha_j^m\mathcal{N}\left(\mathcal{I}_l,\mathcal{I}_m\right)+\alpha_i^l \beta_j^m \mathcal{N}\left(\mathcal{I}_l, \mathcal{J}_m\right)\right.\\
&\left.+\beta_i^l \alpha_j^m \mathcal{N}\left(\mathcal{J}_l, \mathcal{I}_m\right)+\beta_i^l\beta_j^m\mathcal{N}\left(\mathcal{J}_l, \mathcal{J}_m\right)\right).
\end{split}\end{equation}
By Theorem~\ref{pullup001}, all terms in \eqref{t01} vanish except possibly the diagonal mixed terms. Write
\[
\Theta:=\mathcal N(\mathcal I_1,\mathcal J_1)
=\mathcal N(\mathcal I_2,\mathcal J_2)
=\mathcal N(\mathcal I_3,\mathcal J_3).
\]
Using the symmetry of the mixed Nijenhuis concomitant, \eqref{t01} reduces to
\[
\mathcal N(\mathcal K_i,\mathcal K_j)
=
\left(\sum_{l=1}^3
\bigl(\alpha_i^l\beta_j^l+\beta_i^l\alpha_j^l\bigr)\right)\Theta.
\]
Since the rows of $T$ and $S$ are orthonormal,
\[
\sum_{l=1}^3
\bigl(\alpha_i^l\beta_j^l+\beta_i^l\alpha_j^l\bigr)
=
\frac12\left(\sum_{l=1}^3t_{il}t_{jl}
-\sum_{l=1}^3s_{il}s_{jl}\right)
=
\frac12(\delta_{ij}-\delta_{ij})=0.
\]
Therefore
\begin{equation}\label{kjvd11}
\mathcal{N}(\mathcal{K}_i,\mathcal{K}_j)=0,\qquad i,j=1,2,3.
\end{equation}
Thus $\mathcal{K}_1,\mathcal{K}_2,\mathcal{K}_3$ are simultaneously integrable. Together with \eqref{ttnnb00}--\eqref{ttnnb01}, this shows that
$\left(\mathcal{K}_1,\mathcal{K}_2,\mathcal{K}_3\right)$ forms a generalized Clifford structure.
\qed

Before proving Theorem~1.3, we record the differential identity that controls the parameter dependence of the family $\hat{\mathcal I}(\zeta_1,\zeta_2)$.  This allows us to verify integrability directly through the generalized Nijenhuis tensor.

{\it Proof of Theorem \ref{lbnl1}}.
Let
\[
S:=S^2_{\zeta_1}\times S^2_{\zeta_2}
\simeq \mathbb{CP}^1\times\mathbb{CP}^1,
\qquad
Z=M\times S.
\]
Then
\[
\mathbb T Z\simeq\mathbb T M\oplus\mathbb T S.
\]
We use the Dorfman bracket
\begin{equation}\label{eq:Dorfman-twistor}
[X+\alpha,Y+\beta]_D
=
[X,Y]+\mathcal L_X\beta-\iota_Yd\alpha,
\end{equation}
where $X,Y\in\Gamma(TZ)$ and $\alpha,\beta\in\Omega^1(Z)$.

Write
\begin{equation}\label{eq:Ihat-family-twistor}
\hat{\mathcal I}(\zeta_1,\zeta_2)
=
\vec c(\zeta_1)\cdot\vec{\mathcal I}^{(+)}
+
\vec d(\zeta_2)\cdot\vec{\mathcal I}^{(-)},
\end{equation}
where
$
\vec{\mathcal I}^{(\pm)}
=
(\mathcal I_1^{(\pm)},\mathcal I_2^{(\pm)},\mathcal I_3^{(\pm)})$
and
\begin{equation}\label{eq:stereographic-vectors-twistor}
\begin{split}
\vec c(\zeta_1)
&=
\left(
\frac{1-|\zeta_1|^2}{1+|\zeta_1|^2},
\frac{\mathbf i(\bar\zeta_1-\zeta_1)}{1+|\zeta_1|^2},
-\frac{\zeta_1+\bar\zeta_1}{1+|\zeta_1|^2}
\right),\\
\vec d(\zeta_2)
&=
\left(
\frac{1-|\zeta_2|^2}{1+|\zeta_2|^2},
\frac{\mathbf i(\bar\zeta_2-\zeta_2)}{1+|\zeta_2|^2},
-\frac{\zeta_2+\bar\zeta_2}{1+|\zeta_2|^2}
\right).
\end{split}
\end{equation}

We first record the parameter-space identity that will be used below.
Let
\[
Y=(Y_1,Y_2)\in TS,\qquad
u=d\vec c(Y_1),\qquad
v=d\vec d(Y_2).
\]
Since $\vec c$ and $\vec d$ take values in the unit sphere,
\[
u\cdot\vec c=0,\qquad v\cdot\vec d=0.
\]
Under the stereographic identification
$\mathbb{CP}^1\simeq S^2$, the canonical complex structure
$J_{\mathbb{CP}^1}$ corresponds to the standard complex structure on
the unit sphere given by
\[
J^{S^2}_{p}(w)=p\times w,
\qquad
w\in T_pS^2.
\]
Consequently, for the stereographic parametrizations
$\vec c$ and $\vec d$ used above,
\[
d\vec c(J_SY_1)
=
\vec c\times d\vec c(Y_1),
\qquad
d\vec d(J_SY_2)
=
\vec d\times d\vec d(Y_2).
\]
Using the quaternionic multiplication rule for the two sectors and the
fact that the $(+)$ and $(-)$ sectors annihilate each other, we obtain
\begin{align*}
(d_Y\hat{\mathcal I})\hat{\mathcal I}
&=
(u\times\vec c)\cdot\vec{\mathcal I}^{(+)}
+
(v\times\vec d)\cdot\vec{\mathcal I}^{(-)}\\
&=
-(\vec c\times u)\cdot\vec{\mathcal I}^{(+)}
-(\vec d\times v)\cdot\vec{\mathcal I}^{(-)}\\
&=
-d_{J_SY}\hat{\mathcal I}.
\end{align*}
Thus
\begin{equation}\label{eq:twistor-CR}
(d_Y\hat{\mathcal I})\hat{\mathcal I}
=
-d_{J_SY}\hat{\mathcal I}.
\end{equation}
Since $\hat{\mathcal I}^2=-\operatorname{Id}$, differentiation gives
\[
(d_Y\hat{\mathcal I})\hat{\mathcal I}
+
\hat{\mathcal I}(d_Y\hat{\mathcal I})=0.
\]
Combining this with \eqref{eq:twistor-CR}, we also have
\begin{equation}\label{eq:twistor-CR2}
\hat{\mathcal I}(d_Y\hat{\mathcal I})
=
d_{J_SY}\hat{\mathcal I}.
\end{equation}

We now prove that the generalized Nijenhuis tensor
$N_{\hat{\mathbb I}}$ vanishes. Since $N_{\hat{\mathbb I}}$ is tensorial,
it is enough to work at a fixed point $(p,s)\in M\times S$ and choose
convenient local extensions. In particular, sections
$A,B\in\mathbb T_pM$ may be extended locally so that they are independent
of the $S$--variables.

According to the decomposition
$\mathbb T Z=\mathbb T M\oplus\mathbb T S,$
it suffices to consider the following three types of terms.

\medskip
\noindent\textbf{(1) The $S$--$S$ term.}
For $\alpha,\beta\in\mathbb T S$, the restriction of
$\hat{\mathbb I}$ to $\mathbb T S$ is $\mathcal J$. Hence
\begin{equation}\label{key001}
N_{\hat{\mathbb I}}(\alpha,\beta)
=
N_{\mathcal J}(\alpha,\beta)=0,
\end{equation}
because $J_S$ is integrable.

\medskip
\noindent\textbf{(2) The mixed $S$--$M$ term.}
Let $A=X+\xi\in\mathbb T M$ be chosen independent of the
$S$--variables, and let $Y\in TS$. Then
$[Y,A]_D=[J_SY,A]_D=0.$
To see explicitly how the parameter dependence of $\hat{\mathcal I}$
enters the Dorfman bracket, write locally
$\hat{\mathcal I}
=
\begin{pmatrix}
P&Q\\
R&T
\end{pmatrix}.$
Then
$
\hat{\mathcal I}A
=
(PX+Q\xi)+(RX+T\xi).
$
By \eqref{eq:Dorfman-twistor}, we have
\begin{align*}
[Y,\hat{\mathcal I}A]_D
&=
[Y,PX+Q\xi]+\mathcal L_Y(RX+T\xi)\\
&=
(d_YP)X+(d_YQ)\xi+(d_YR)X+(d_YT)\xi\\
&=
(d_Y\hat{\mathcal I})A,
\end{align*}
because $X$ and $\xi$ are independent of the $S$--variables. Similarly,
\[
[J_SY,\hat{\mathcal I}A]_D
=
(d_{J_SY}\hat{\mathcal I})A.
\]
Since $\hat{\mathbb I}Y=J_SY$, the definition of the generalized
Nijenhuis tensor yields
\begin{equation}\label{key002a}\begin{split}
N_{\hat{\mathbb I}}(Y,A)
&=
[J_SY,\hat{\mathcal I}A]_D
-\hat{\mathbb I}[J_SY,A]_D
-\hat{\mathbb I}[Y,\hat{\mathcal I}A]_D
-[Y,A]_D\\
&=
(d_{J_SY}\hat{\mathcal I})A
-\hat{\mathcal I}(d_Y\hat{\mathcal I})A\\
&=0
\end{split}\end{equation}
by \eqref{eq:twistor-CR2}.
For the reversed order, the Dorfman bracket gives
\[
[\hat{\mathcal I}A,Y]_D
=-(d_Y\hat{\mathcal I})A,
\qquad
[\hat{\mathcal I}A,J_SY]_D
=-(d_{J_SY}\hat{\mathcal I})A,
\]
while
\[
[A,Y]_D=[A,J_SY]_D=0.
\]
Therefore
\begin{equation}\label{key002b}\begin{split}
N_{\hat{\mathbb I}}(A,Y)
&=
[\hat{\mathcal I}A,J_SY]_D
-\hat{\mathbb I}[\hat{\mathcal I}A,Y]_D
-\hat{\mathbb I}[A,J_SY]_D
-[A,Y]_D\\
&=
-(d_{J_SY}\hat{\mathcal I})A
+\hat{\mathcal I}(d_Y\hat{\mathcal I})A\\
&=0
\end{split}\end{equation}
by \eqref{eq:twistor-CR2}.

If $\eta\in T^*S$ and $A\in\mathbb T M$, then $\eta$ and
$\hat{\mathbb I}\eta=-J_S^T\eta$ are pulled back from $S$. Since the
anchors of $A$ and $\hat{\mathcal I}A$ are tangent to $M$, while
$d\eta$ and $d(J_S^T\eta)$ have only $S$--components, the Dorfman formula
shows directly that
\begin{equation}\label{key003}
N_{\hat{\mathbb I}}(\eta,A)
=
N_{\hat{\mathbb I}}(A,\eta)=0.
\end{equation}
Thus all mixed terms vanish.

\medskip
\noindent\textbf{(3) The $M$--$M$ term.}
Let $A=X_A+\xi_A,\
B=X_B+\xi_B$
be independent of the $S$--variables. For every fixed
$s=(\zeta_1,\zeta_2)\in S$, Theorem~\ref{6666661} gives
\begin{equation}\label{key004}
N_{\hat{\mathcal I}(s)}(A,B)=0.
\end{equation}
Hence the $\mathbb T M$--component of
$N_{\hat{\mathbb I}}(A,B)$ vanishes.

It remains to check the possible $T^*S$--component caused by the
$S$--dependence of $\hat{\mathcal I}$. Let
$a=U+\mu,\  b=V+\nu$
be sections of the pull-back of $\mathbb T M$, and let $W\in TS$.
Evaluating the Dorfman bracket on $W$ gives
\begin{align}
\bigl([a,b]_D\bigr)_{T^*S}(W)
&=
(\mathcal L_U\nu-\iota_Vd\mu)(W)\notag\\
&=
\nu(d_WU)+(d_W\mu)(V)\notag\\
&=
2\langle d_Wa,b\rangle .
\label{eq:vertical-bracket-S}
\end{align}
Applying \eqref{eq:vertical-bracket-S} to the four terms in the
generalized Nijenhuis tensor, and using that $A$ and $B$ are independent
of $S$, gives
\begin{align*}
\bigl(N_{\hat{\mathbb I}}(A,B)\bigr)_{T^*S}(W)
={}&
2\left\langle
(d_W\hat{\mathcal I})A,\hat{\mathcal I}B
\right\rangle+
2\left\langle
(d_{J_SW}\hat{\mathcal I})A,B
\right\rangle .
\end{align*}
Since $\hat{\mathcal I}$ is orthogonal and
$\hat{\mathcal I}^2=-\operatorname{Id}$,
\[
\left\langle
(d_W\hat{\mathcal I})A,\hat{\mathcal I}B
\right\rangle
=
\left\langle
(d_W\hat{\mathcal I})\hat{\mathcal I}A,B
\right\rangle .
\]
Using \eqref{eq:twistor-CR}, we obtain
\[
\left\langle
(d_W\hat{\mathcal I})A,\hat{\mathcal I}B
\right\rangle
=
-\left\langle
(d_{J_SW}\hat{\mathcal I})A,B
\right\rangle .
\]
Therefore
\[
\bigl(N_{\hat{\mathbb I}}(A,B)\bigr)_{T^*S}=0,
\]
and hence
\begin{equation}\label{key005}
N_{\hat{\mathbb I}}(A,B)=0.
\end{equation}

Combining the three cases, (\ref{key001})-(\ref{key005}) yield that
\[
N_{\hat{\mathbb I}}=0.
\]
Therefore $\hat{\mathbb I}$ is integrable.
\qed

\section{T-duality and rank-3 generalized Clifford structures}
The formulation of T-duality used here follows the generalized-geometric approach of
Cavalcanti and Gualtieri \cite{CavalcantiGualtieri2010TD}, building on the topological
T-duality framework for circle and torus bundles with $H$-flux developed in
\cite{BouwknegtEvslinMathai2004,BouwknegtHannabussMathai2005}.
In this section, we show that rank-3 twisted generalized Clifford structures are stable under
T-duality. More precisely, we prove that T-duality preserves not only the original Clifford triple
$(\mathcal{I}_1,\mathcal{I}_2,\mathcal{I}_3)$, but also the induced structures $(\mathcal{J}_1,\mathcal{J}_2,\mathcal{J}_3,\mathcal{G})$ and the associated
Spin(3)-rotated family $(\mathcal{K}_1(\zeta_1,\zeta_2),\mathcal{K}_2(\zeta_1,\zeta_2),\mathcal{K}_3(\zeta_1,\zeta_2))$.
This shows that the Clifford data and the rotated family underlying the twistor construction
are compatible with the natural symmetry of generalized geometry given by T-duality.

Let $(M,H)$ and $(\widetilde M,\widetilde H)$ be a pair of T-dual manifolds. We shall use the
following standard formulation of T-duality in generalized geometry: there exists an orthogonal
bundle isomorphism
\[
\Phi:\mathbb{T}M=TM\oplus T^*M \longrightarrow \mathbb{T}\widetilde M
  =T\widetilde M\oplus T^*\widetilde M
\]
such that
\begin{equation}\label{TD-1}
\langle \Phi(A),\Phi(B)\rangle=\langle A,B\rangle,
\end{equation}
and
\begin{equation}\label{TD-2}
\Phi\bigl([ A,B]_H\bigr)=[\Phi(A),\Phi(B)]_{\widetilde H},
\qquad  A,B\in \Gamma(\mathbb{T}M).
\end{equation}
Here $[\cdot,\cdot]_H$ and $[\cdot,\cdot]_{\widetilde H}$ denote the $H$-twisted and
$\widetilde H$-twisted Dorfman brackets, respectively.

Condition \eqref{TD-2} implies that any generalized-geometric structure defined in terms of the
natural pairing and the twisted Dorfman bracket can be transported through $\Phi$.
In particular, for any bundle endomorphism $\mathcal{I}\in \Gamma(End(\mathbb{T}M))$, we define its
T-dual transform by
\[
\widetilde{\mathcal{I}}=\Phi \mathcal{I}\Phi^{-1}\in \Gamma(End(\mathbb{T}\widetilde M)).
\]

\begin{lem}\label{lem:TD-Nijenhuis}
Let $\mathcal{I},\mathcal{J}\in \Gamma(End(\mathbb{T}M))$. Then the mixed generalized Nijenhuis tensor is
preserved by T-duality in the sense that
\[
\mathcal N_{\widetilde H}(\widetilde{\mathcal{I}},\widetilde {\mathcal{J}})(\Phi A,\Phi B)
=
\Phi\Bigl(\mathcal N_H(\mathcal{I},\mathcal{J})(A,B)\Bigr),
\qquad A,B\in \Gamma(\mathbb{T}M).
\]
In particular,
\[
\mathcal N_H(\mathcal{I},\mathcal{J})=0 \quad \Longrightarrow\quad
\mathcal N_{\widetilde H}(\widetilde{\mathcal{I}},\widetilde {\mathcal{J}})=0.
\]
\end{lem}

\begin{proof}
This follows directly from the definition of the generalized Nijenhuis concomitant and the fact
that $\Phi$ intertwines the twisted Dorfman brackets:
\[
\Phi([A,B]_H)=[\Phi(A),\Phi(B)]_{\widetilde H}.
\]
Indeed, every term in the definition of $\mathcal N(\mathcal{I},\mathcal{J})$ is functorial under conjugation by
$\Phi$, and therefore the whole expression is transported to the corresponding one on
$\mathbb{T}\widetilde M$.
\end{proof}

We now show that T-duality preserves the class of rank-3 twisted generalized Clifford manifolds.
\begin{prop}\label{thm:TD-Clifford}
Let $(M,H)$ and $(\widetilde M,\widetilde H)$ be T-dual manifolds, and let
\[
\Phi:(\mathbb{T}M,[\cdot,\cdot]_H)\longrightarrow
(\mathbb{T}\widetilde M,[\cdot,\cdot]_{\widetilde H})
\]
be an orthogonal Courant algebroid isomorphism.
Assume that $(\mathcal{I}_1,\mathcal{I}_2,\mathcal{I}_3)$ is an $H$-twisted rank-3 generalized Clifford structure on $M$.
Define
\[
\widetilde{\mathcal{I}}_a:=\Phi \mathcal{I}_a\Phi^{-1},\qquad a=1,2,3.
\]
Then $(\widetilde{\mathcal I}_1,\widetilde{\mathcal I}_2,\widetilde{\mathcal I}_3)$ is a
$\widetilde H$-twisted rank-3 generalized Clifford structure on $\widetilde M$.
\end{prop}

\begin{proof}
Since each $\mathcal{I}_a$ is orthogonal and satisfies $\mathcal{I}_a^2=-{\rm Id}$, and since $\Phi$ is orthogonal, each
$\widetilde{\mathcal I}_a$ is again an orthogonal almost generalized complex structure on
$\mathbb{T}\widetilde M$:
\[
\widetilde{\mathcal{I}}_a^2=\Phi \mathcal{I}_a^2\Phi^{-1}=-{\rm Id}.
\]
Moreover, the Clifford relations are preserved:
\[
\widetilde{\mathcal{I}}_a\widetilde{\mathcal{I}}_b+\widetilde{\mathcal{I}}_b\widetilde{\mathcal{I}}_a
=
\Phi(\mathcal{I}_a\mathcal{I}_b+\mathcal{I}_b\mathcal{I}_a)\Phi^{-1}
=
-2\delta_{ab}{\rm Id}.
\]
Finally, since each $\mathcal I_a$ is integrable with respect to the $H$-twisted Dorfman bracket,
Lemma~\ref{lem:TD-Nijenhuis} gives
\[
\mathcal N_{\widetilde H}(\widetilde{\mathcal I}_a,\widetilde{\mathcal I}_a)=0,
\qquad a=1,2,3.
\]
Hence $(\widetilde{\mathcal I}_1,\widetilde{\mathcal I}_2,\widetilde{\mathcal I}_3)$ defines a
$\widetilde H$-twisted rank-3 generalized Clifford structure on $\widetilde M$.
\end{proof}

We now show that these induced structures \eqref{coincide00000} are also compatible with T-duality.
\begin{prop}\label{prop:TD-induced}
With the notation \eqref{coincide00000}, define
\begin{equation}\label{coincide000001}
\widetilde{\mathcal{J}}_i:=\frac12\varepsilon_{ijk}\widetilde{\mathcal{I}}_j\widetilde{\mathcal{I}}_k,
\qquad
\widetilde{\mathcal{G}}:=-\widetilde{\mathcal{I}}_1\widetilde{\mathcal{I}}_2\widetilde{\mathcal{I}}_3.
\end{equation}
Then
\[
\widetilde{\mathcal{J}}_i=\Phi \mathcal{J}_i\Phi^{-1},
\qquad
\widetilde{\mathcal{G}}=\Phi \mathcal{G}\Phi^{-1}.
\]
\end{prop}

\begin{proof}
Using the definition of $\mathcal{J}_i$ in \eqref{coincide00000} and the multiplicativity of conjugation, we obtain
\[
\widetilde{\mathcal{J}}_i
=
\frac12\varepsilon_{ijk}\widetilde{\mathcal{I}}_j\widetilde{\mathcal{I}}_k
=
\frac12\varepsilon_{ijk}(\Phi \mathcal{I}_j\Phi^{-1})(\Phi \mathcal{I}_k\Phi^{-1})
=
\Phi\Bigl(\frac12\varepsilon_{ijk}\mathcal{I}_j\mathcal{I}_k\Bigr)\Phi^{-1}
=
\Phi \mathcal{J}_i\Phi^{-1}.
\]
Similarly,
\[
\widetilde{\mathcal{G}}=-\widetilde{\mathcal{I}}_1\widetilde{\mathcal{I}}_2\widetilde{\mathcal{I}}_3
=
-\Phi \mathcal{I}_1\mathcal{I}_2\mathcal{I}_3\Phi^{-1}
=
\Phi \mathcal{G}\Phi^{-1}.
\]
\end{proof}
As an immediate consequence, all algebraic identities in \eqref{coincide} are preserved by T-duality.
Define similarly
\[
\widetilde{\mathcal{I}}_i^\pm:=\frac12(\widetilde{\mathcal{J}}_i\pm \widetilde{\mathcal{I}}_i),
\qquad
\widetilde{\mathcal{G}}^\pm:=\frac12(Id\pm \widetilde{\mathcal{G}}),
\]
and let
\[
\bigl(\widetilde{\mathcal{K}}_1(\zeta_1,\zeta_2),\widetilde{\mathcal{K}}_2(\zeta_1,\zeta_2),
\widetilde{\mathcal{K}}_3(\zeta_1,\zeta_2)\bigr)
\]
be the corresponding Clifford-rotated family on $\widetilde M$.

\begin{prop}\label{prop:TD-rotation}
For every $(\zeta_1,\zeta_2)\in \mathbb{CP}^1\times \mathbb{CP}^1$, one has
\[
\widetilde{\mathcal{K}}_i(\zeta_1,\zeta_2)=\Phi {\mathcal{K}}_i(\zeta_1,\zeta_2)\Phi^{-1},
\qquad i=1,2,3.
\]
Equivalently, T-duality commutes with the rank-3 Clifford rotations.
\end{prop}

\begin{proof}
By Proposition~\ref{prop:TD-induced}, we have
\[
\widetilde{\mathcal{I}}_i^\pm
=
\frac12(\widetilde{\mathcal{J}}_i\pm \widetilde{\mathcal{I}}_i)
=
\frac12\Phi({\mathcal{J}}_i\pm{\mathcal{I}}_i)\Phi^{-1}
=
\Phi {\mathcal{I}}_i^\pm \Phi^{-1}.
\]
Hence, using the definition of the rotated structures in Section~4,
\[
\widehat{\widetilde{\mathcal{I}}}_i^{+}
=
\sum_{l=1}^3 t_{il}\widetilde{\mathcal{I}}_l^+
=
\sum_{l=1}^3 t_{il}\Phi \mathcal I_l^+\Phi^{-1}
=
\Phi\Bigl(\sum_{l=1}^3 t_{il}{\mathcal{I}}_l^+\Bigr)\Phi^{-1}
=
\Phi\widehat{\mathcal{I}}_i^{\, +}\Phi^{-1},
\]
and similarly
\[
\widehat{\widetilde{\mathcal{I}}}_i^{-}
=
\Phi\widehat{\mathcal{I}}_i^{\, -}\Phi^{-1},
\]
Therefore,
\[
\widetilde{\mathcal K}_i
=
\widehat{\widetilde{\mathcal I}}_i^{\, +}
+\widehat{\widetilde{\mathcal I}}_i^{\, -}
=
\Phi\Bigl(\widehat{{\mathcal{I}}}_i^{+}+\widehat{{\mathcal{I}}}_i^{-}\Bigr)\Phi^{-1}
=
\Phi \mathcal{K}_i\Phi^{-1}.
\]
This proves the claim.
\end{proof}

Combining Proposition~\ref{thm:TD-Clifford} with Proposition~\ref{prop:TD-rotation}, we obtain:

\begin{cor}\label{cor:TD-twistor-family}
The T-dual of the Spin(3)-rotated family associated with a rank-3 twisted generalized Clifford
structure is exactly the Spin(3)-rotated family associated with the T-dual Clifford triple.
In particular, the construction of the $S^2\times S^2$-family of generalized complex structures
is natural with respect to T-duality.
\end{cor}

\begin{rem}
Proposition~\ref{thm:TD-Clifford} and Proposition~\ref{prop:TD-rotation} show that the
$M$-component of the twistor generalized complex structure constructed in Theorem~1.3 is
compatible with T-duality. It is therefore natural to expect that, after pulling back the twist to
$M\times S^2\times S^2$, the corresponding generalized complex structure on the twistor space
admits a T-dual interpretation as well. We leave this total-space formulation to future work.
\end{rem}



\end{document}